\documentclass{article}
\usepackage{amsmath, amsfonts, amsthm, amssymb, color}
\usepackage{graphicx}
\usepackage{float}
\usepackage{verbatim}

\allowdisplaybreaks

\numberwithin{equation}{section}

\newtheorem{lemma}{Lemma}[section]

\newtheorem{prop}[lemma]{Proposition}
\newtheorem{thm}[lemma]{Theorem}
\newtheorem{cor}[lemma]{Corollary}
\theoremstyle{definition}

\newtheorem{example}[lemma]{Example}
\newtheorem{conjecture}[lemma]{Conjecture}

\theoremstyle{remark}
\newtheorem{remark}[lemma]{Remark}

\def\R{\mathbb{R}}
\def\N{\mathbb{N}}
\def\D{\mathcal{D}}
\def\L{\mathcal{L}}
\def\O{\mathcal{O}}
\numberwithin{equation}{section} \numberwithin{table}{section}

\title{An analogue of Khintchine's theorem for self-conformal sets}
\author{Simon Baker\\ \\
\emph{Mathematics institute,} \\ \emph{University of Warwick,} \\ \emph{Coventry,  CV4 7AL, UK.} \\ Email: simonbaker412@gmail.com\\}

\date{\today}
\begin{document}
\maketitle


\begin{abstract}
Khintchine's theorem is a classical result from metric number theory which relates the Lebesgue measure of certain limsup sets with the convergence/divergence of naturally occurring volume sums. In this paper we ask whether an analogous result holds for iterated function systems (IFSs). We say that an IFS is approximation regular if we observe Khintchine type behaviour, i.e., if the size of certain limsup sets defined using the IFS is determined by the convergence/divergence of naturally occurring sums. We prove that an IFS is approximation regular if it consists of conformal mappings and satisfies the open set condition. The divergence condition we introduce incorporates the inhomogeneity present within the IFS. We demonstrate via an example that such an approach is essential. We also formulate an analogue of the Duffin-Schaeffer conjecture and show that it holds for a set of full Hausdorff dimension.

Combining our results with the mass transference principle of Beresnevich and Velani \cite{BerVel}, we prove a general result that implies the existence of exceptional points within the attractor of our IFS. These points are exceptional in the sense that they are ``very well approximated". As a corollary of this result, we obtain a general solution to a problem of Mahler, and prove that there are badly approximable numbers that are very well approximated by quadratic irrationals.

The ideas put forward in this paper are introduced in the general setting of iterated function systems that may contain overlaps. We believe that by viewing iterated function systems from the perspective of metric number theory, one can gain a greater insight into the extent to which they overlap. The results of this paper should be interpreted as a first step in this investigation.\footnote{This research was supported by EPSRC grant EP/M001903/1.}\\


\end{abstract}

\section{Introduction}
Let $V\subset \R^{d}$ be a closed set. A map $\phi:V\to V$ is called a contraction if there exists $r\in(0,1)$ such that $|\phi(x)-\phi(y)|\leq r|x-y|$ for all $x,y\in V$. An \emph{iterated function system} (IFS) on $V$ is a finite set of contractions $\Phi=\{\phi_{i}\}_{i\in \D}.$ A well known result due to Hutchinson \cite{Hut} states that for any iterated function system $\Phi,$ there exists a unique non-empty compact set $X\subset \R^{d}$ such that
\begin{equation}
\label{Hutchinson's equation}
X=\bigcup_{i\in\D}\phi_{i}(X).
\end{equation}The set $X$ is called the \emph{attractor} associated to $\Phi$. Often one is interested in understanding the geometric properties of $X$. When the images of $X$ under the elements of $\Phi$ are well separated, then the geometry of $X$ is well understood. Moreover, under this assumption one can often calculate the dimension of $X$ for the various notions of dimension. However, when the images of $X$ overlap significantly, the geometry of $X$ is much more complicated. One cannot draw as many conclusions as in the non-overlapping case. This being said, one of the guiding principles within fractal geometry is that if the IFS generating $X$ has no obvious mechanism preventing $X$ from satisfying a certain property, then that property should be satisfied. Properties we might be interested in include whether the dimension of $X$ satisfies a certain formula, whether the dimension of certain measures supported on $X$ satisfy a certain formula, whether certain measures supported on $X$ are absolutely continuous with respect to the Lebesgue measure, etc. Consequently, it is believed that despite the presence of significant overlaps within an IFS there should be much that we can say. For more on IFS's with overlaps we refer the reader to \cite{Feng,Hoc,PerSol} and the references therein. In this paper we view IFS's from the perspective of metric number theory. We now take the opportunity to recall the relevant background from this area.

Given $\Psi:\mathbb{N}\to\mathbb{R}_{\geq 0},$ we associate the set
\begin{align*}
J(\Psi):=\Big\{x\in\mathbb{R}: |x-p/q|\leq \Psi(q)\, \textrm{for i.m. } (p,q)\in\mathbb{Z}\times\mathbb{N}\Big\}.
\end{align*}Here and throughout i.m. will be used as a shorthand for infinitely many. The following well known theorem is due to Khintchine \cite{Khit}.
\begin{thm}\label{Khintchine's theorem}(Khintchine's theorem)
 If $\Psi:\N\to\R_{\geq 0}$ is a decreasing function and $$\sum_{q=1}^{\infty}q\Psi(q)=\infty,$$ then Lebesgue almost every $x$ is an element of $J(\Psi)$.
\end{thm} This result is complemented by the following straightforward consequence of the Borel-Cantelli lemma.
\begin{thm}
If $\sum_{q=1}^{\infty}q\Psi(q)<\infty,$ then $J(\Psi)$ has zero Lebesgue measure.
\end{thm}
Note that by an example of Duffin and Schaeffer \cite{DufSch} the monotonicity assumption appearing in Theorem \ref{Khintchine's theorem} cannot be removed entirely. In response to this example they proposed the following refinement of Theorem \ref{Khintchine's theorem}, now known as the Duffin-Schaeffer conjecture.

\begin{conjecture}(Duffin-Schaeffer conjecture)
\label{DufSch conjecture}
If $\Psi:\mathbb{N}\to\mathbb{R}_{\geq 0}$ is an arbitrary function satisfying
$$\sum_{q=1}^{\infty}\#\{1\leq p\leq q: \gcd(p,q)=1\}\Psi(q)=\infty,$$ then Lebesgue almost every $x$ is an element of $J(\Psi)$.
\end{conjecture}

Our IFS analogue of the set $J(\Psi)$ is defined as follows. Given an IFS $\Phi$ let $\D^n:=\{(i_1,\ldots,i_n):i_j\in\D\textrm{ for }1\leq j \leq n\}$ and $\D^{*}:=\cup_{n=1}^{\infty}\mathcal{D}^n$. Fix $z\in X$ and let $\Psi:\D^{*}\to\mathbb{R}_{\geq 0}$. The analogue of $J(\Psi)$ is $$W(\Psi,z):=\Big\{x\in X: x\in B(\phi_{I}(z), \Psi(I)) \textrm{ for i.m. }I\in\D^{*}\Big\}.$$ Here and throughout we will use $I=(i_1,\ldots,i_n)$ to denote an element of $\D^*,$ and $\phi_{I}$ will denote the concatenation $\phi_{i_1}\circ\cdots \circ\phi_{i_{n}}$. We will denote the length of a finite word $I$ by $|I|$ and let $X_I$ denote $\phi_{I}(X)$. Throughout this paper we will refer to $X_I$ as a \emph{cylinder set} and say that it has \emph{rank} $|I|$. We call any function $\Psi:\D^{*}\to\mathbb{R}_{\geq 0}$ an \emph{approximating function}.

Note that the set $W(\Psi,z)$ has some key differences when compared with the set $J(\Psi)$. First of all, we have introduced the additional variable $z$. Secondly, the function $\Psi$ may depend on the specific digits of $I$ not just the length of $I$. As we will see in Example \ref{scaling example}, allowing $\Psi$ to depend on the specific digits of $I$ and not just the length of $I$ is essential.

Theorem \ref{Khintchine's theorem} quantifies the good distributional properties the rational numbers have within $\mathbb{R}$. Similarly, if an analogue of Khintchine's theorem were to hold for $W(\Psi,z),$ this would reflect the good distributional properties the images of $z$ have within $X$. Therefore, if an analogue of Khintchine's theorem were to hold for every $z\in X,$ this would tell us that from the perspective of metric number theory our IFS does not contain significant overlaps. One could then ask whether the presence of a Khintchine type theorem would imply any other nice properties for our IFS. These ideas are the motivation behind this paper.

With the set $W(\Psi,z)$ defined as above for a general IFS, it isn't obvious what the appropriate analogue of Khintchine's theorem should be. In this paper we narrow our scope to determining an analogue of Khintchine's theorem for IFS's consisting of conformal mappings. Restricting to IFS's consisting of conformal mappings we introduce the notion of an approximation regular IFS. Put simply, an IFS will be approximation regular if we observe Khintchine type behaviour, i.e., if the size of $W(\Psi,z)$ is determined by the convergence/divergence of naturally occurring sums. We prove that an IFS consisting of conformal mappings is approximation regular if it satisfies the open set condition. These results are of independent interest but should also be interpreted as a first step in our investigation into studying overlapping attractors from the perspective of metric number theory. In the final section of this paper we prove a complementary result which states that whenever our IFS contains an exact overlap then it cannot be approximation regular. This is relevant as the standard mechanism by which one can construct an IFS that fails to satisfy a certain property, that we would otherwise expect to be true, is to construct the IFS in such a way that it contains an exact overlap.
\\

\noindent \textbf{Notation.} Throughout this paper we make use of the standard big $\mathcal{O}$ notation. Given two positive real valued functions $f,g$ defined on some set $S,$ we write $f\asymp g$ if there exists a positive constant $C$ such that $C^{-1}f(x)\leq g(x)\leq Cf(x)$ for all $x\in S$.

\subsection{Self-similar sets and self-conformal sets}
In this section we describe the attractors we will be focusing on. Suppose $\phi:\mathbb{R}^{d}\to\mathbb{R}^{d}$ is a contraction and that it also satisfies the condition $|\phi(x)-\phi(y)|=r|x-y|$ for all $x,y\in\R^{d},$ for some $r\in(0,1)$. If $\phi$ satisfies this condition we call $\phi$ a \emph{similarity}. When our IFS $\Phi$ consists solely of similarities, we say that $X$ is a \emph{self-similar set}. The unique $s$ for which
\begin{equation}
\label{similarity dimension}
\sum_{i\in \D}r_i^s=1
\end{equation} is called the \emph{similarity dimension}. We will denote the similarity dimension by $\dim_{S}(X)$. This choice of notation is a little misleading as the similarity dimension is a function of $\Phi$ not $X$. There maybe several IFS's with different similarity dimensions but each with $X$ as their attractor. However, it should always be clear from our context which IFS we are referring to.

Self-conformal sets are a natural generalisation of self-similar sets. Let $V\subset \R^d$ be an open set, a $C^{1}$ map $\phi:V\to \R^{d}$ is a \emph{conformal mapping} if it preserves angles. Equivalently $\phi$ is a conformal mapping if the differential $\phi'$ satisfies $|\phi'(x)y|=|\phi'(x)||y|$ for all $x\in V$ and $y\in \R^{d}$. We call $\Phi$ a \emph{conformal iterated function system} on a compact set $Y\subset \R^d,$ if each $\phi_i$ can be extended to an injective conformal contraction on some open connected neighbourhood $V$ that contains $Y$ and $0<\inf_{x\in V} |\phi_i'(x)|\leq \sup_{x\in V}|\phi_i'(x)|<1$. Throughout this paper we will also assume that the differentials are H\"{o}lder continuous, i.e., there exists $\alpha>0$ and $c>0$ such that $$||\phi_i'(x)|-|\phi_i'(y)||\leq c|x-y|^\alpha$$ for all $x,y\in V$. When $\Phi$ is a conformal iterated function system we call the attractor $X$ a \emph{self-conformal set}. The generalisation of the similarity dimension within the setting of self-conformal sets is the unique zero of Bowen's equation, that is the unique $s\in \mathbb{R}$ that satisfies $P(s)=0,$ where
\begin{equation}
\label{Pressure dimension}
P(s):=\lim_{n\to\infty}\frac{1}{n}\log \inf_{x\in X} \sum_{I\in \D^{n}}|\phi_{I}'(x)|^{s}=\lim_{n\to\infty}\frac{1}{n}\log \sup_{x\in X} \sum_{I\in \D^{n}}|\phi_{I}'(x)|^{s}.
\end{equation} When each element of $\Phi$ is a similarity then the equation $P(s)=0$ reduces to \eqref{similarity dimension}. For notational convenience we will denote the unique $s$ satisfying $P(s)=0$ by $\dim_{S}(X)$ and also call it the similarity dimension.

When the images of $X$ under the elements of $\Phi$ are well separated, the Hausdorff dimension of $X$ is often equal to $\dim_{S}(X)$. Indeed when there exists an open set $O\subset \mathbb{R}^{d}$ such that $\phi_{i}(O)\subset O$ for all $i\in \D,$ and $\phi_{i}(O)\cap \phi_{j}(O)=\emptyset$ for all $i\neq j,$ then $\dim_{H}(X)=\dim_{S}(X)$ and $X$ has positive and finite $\dim_{S}(X)$-dimensional Hausdorff measure. If there exists an open set $O$ satisfying the above criteria then the IFS $\Phi$ is said to satisfy the \emph{open set condition}. These results are well known and date back to the work of Ruelle \cite{Rue}. For a proof see \cite{Fal}. Note that without any separation assumptions we still have the upper bound $\dim_{H}(X)\leq \dim_{S}(X).$

\subsection{Statement of results}
We now introduce the notion of an approximation regular pair and an approximation regular IFS. When defining an approximation regular pair we have to be careful whether the attractor $X$ has zero or positive $\dim_{H}(X)$-dimensional Hausdorff measure. Consequently the following definition is split into two parts.

\begin{itemize}
\item Let $\Phi$ be a conformal iterated function system and suppose $\mathcal{H}^{\dim_{H}(X)}(X)>0.$ Given $z\in X,$ we call $(\Phi,z)$ an \emph{approximation regular pair} if whenever $\theta:\mathbb{N}\to\mathbb{R}_{\geq 0}$ is a decreasing function such that
\begin{equation}
\label{Divergence condition general}
\sum_{n=1}^{\infty} \sum_{I\in\D^{n}} (Diam(X_I)^{\dim_{S}(X)/\dim_{H}(X)}\theta(n))^{\dim_{H}(X)}=\infty,
\end{equation}then $\mathcal{H}^{\dim_{H}(X)}$ almost every $x\in X$ is an element of $W(Diam(X_I)^{\dim_{S}(X)/\dim_{H}(X)}\theta(|I|),z).$
\item Let $\Phi$ be a conformal iterated function system and suppose $\mathcal{H}^{\dim_{H}(X)}(X)=0.$ Given $z\in X,$ we call $(\Phi,z)$ an \emph{approximation regular pair} if whenever $\theta:\mathbb{N}\to\mathbb{R}_{\geq 0}$ is a decreasing function such that \eqref{Divergence condition general} holds, then $\dim_{H}(W(Diam(X_I)^{\dim_{S}(X)/\dim_{H}(X)}\theta(|I|),z))=\dim_{H}(X).$
\end{itemize} We call $\Phi$ an \emph{approximation regular iterated function system} if $(\Phi,z)$ is an approximation regular pair for every $z\in X$.

We emphasise that in the definition of an approximation regular pair the attractor may contain considerable overlaps. So we could have $\dim_{H}(X)<\dim_{S}(X)$. It is also worth noting that originally our sets $W(\Psi,z)$ were defined for any function $\Psi:\D^{*}\to \R_{\geq 0}.$ In the definition of an approximation regular pair we have restricted to functions of the form $\Psi(I)=Diam(X_I)^{\dim_{S}(X)/\dim_{H}(X)}\theta(|I|).$ It is a consequence of this restriction that if $\theta$ is such that \eqref{Divergence condition general} holds, then $\theta$ cannot decay to zero exponentially fast. We have done this because the function $\theta$ depends only on the length of $I,$ and if $\theta$ were to contribute some exponential decay the resulting function would not necessarily take into account the inhomogeneity present within the IFS. Restricting to approximating functions that reflect the inhomogeneity of the IFS is essential. The importance of picking appropriate approximating functions is demonstrated in Section \ref{examples section}.

The approximation regularity of a class of overlapping attractors was studied by the author in \cite{Bak} and \cite{Bakerapprox2}. These attractors are intimately related to the well known Bernoulli convolutions. In \cite{Bak} the author used properties of the Bernoulli convolution to prove approximation regularity results.

Note that when $\Phi$ satisfies the open set condition then \eqref{Divergence condition general} reduces to
\begin{equation}
\label{Divergence condition}
\sum_{n=1}^{\infty} \sum_{I\in\D^{n}} (Diam(X_I)\theta(n))^{\dim_{H}(X)}=\infty.
\end{equation} Moreover $\mathcal{H}^{\dim_{H}(X)}(X)$ will always be positive and finite, so we take the former definition of an approximation regular pair. Under this assumption the relevant limsup sets that appear in the definition of approximation regularity are of the form $W(Diam(X_I)\theta(|I|),z).$ For notational convenience we let $W(\theta,z)$ denote $W(Diam(X_I)\theta(|I|),z)$ throughout.

Our main result is the following.

\begin{thm}
\label{conformal theorem}Let $\Phi$ be an iterated function system with attractor $X$.
\begin{enumerate}
  \item If $\Psi:\D^{*}\to \mathbb{R}_{\geq 0}$ is such that
  \begin{equation}
  \label{boring convergence}
  \sum_{n=1}^{\infty} \sum_{I\in\D^{n}} \Psi(I)^{\dim_{H}(X)}<\infty,
  \end{equation} then $\mathcal{H}^{\dim_{H}(X)}(W(\Psi,z))=0$ for all $z\in X$.
  \item If $\Phi$ is a conformal iterated function system and satisfies the open set condition, then $\Phi$ is an approximation regular IFS.
\end{enumerate}
\end{thm}

In our introduction we mentioned an example of Duffin and Schaeffer \cite{DufSch} which demonstrates that one cannot remove all monotonicity assumptions from the statement of Khintchine's theorem. It is natural to ask whether we can weaken the monotonicity assumptions on $\theta$ appearing in the definition of an approximation regular pair. With this question in mind we make the following conjecture.
\begin{conjecture}
\label{Fractal Duffin Schaeffer}
Theorem \ref{conformal theorem} holds with no underlying monotonicity assumptions on the class of $\theta$.
\end{conjecture}
In \cite{LSV} Levesley, Salp, and Velani studied the approximation properties of balls centred at the left endpoint of the basic intervals generating the middle third cantor set. This fits into our setup with $\Phi=\{x/3,x/3+2/3\}$ and $z=0$. It is a consequence of their work that divergence in \eqref{Divergence condition} is sufficient to prove that $\mathcal{H}^{\dim_{H}(X)}$ almost every $x\in X$ is contained in $W(\theta,0)$ with no monotonicity assumptions on the function $\theta$. A key step in \cite{LSV} was replacing the set $\{\phi_{I}(0)\}_{I\in \D^{n}}$ with its subset $C_{n}$ which consists of those elements of $\{\phi_{I}(0)\}_{I\in \D^{n}}$ with coprime numerator and denominator. Importantly $C_{n}$ has cardinality of the order $c\cdot 2^{n}$ and coprimeness guarantees good separation properties between the set $C_{n}$ and $C_{m}$ for $n\neq m$. These separation properties were important in their proof. In our setup it is not obvious what an appropriate analogue of $C_{n}$ should be and if it even exists. Consequently we cannot prove Conjecture \ref{Fractal Duffin Schaeffer}. We can however prove the following result which doesn't require $\theta$ to be decreasing.

\begin{thm}
\label{Weak Duffin Schaeffer}
 Let $z\in X$ and $\Phi$ be a conformal iterated function system satisfying the open set condition. Suppose $\theta:\mathbb{N}\to\mathbb{R}_{\geq 0}$ is a function satisfying \eqref{Divergence condition} and
\begin{equation}
\label{DufSch Divergence}
\sum_{n=1}^{Q}\theta(n)^{\dim_{H}(X)}\log \frac{1}{\theta(n)} = \O\Big(\Big(\sum_{n=1}^{Q}\theta(n)^{\dim_{H}(X)}\Big)^2\Big).
\end{equation}
Then $\mathcal{H}^{\dim_{H}(X)}$ almost every $x\in X$ is an element of $W(\theta,z).$

\end{thm}

As an application of Theorem \ref{Weak Duffin Schaeffer} we have the following corollary.
\begin{cor}
Let $z\in X$ and $\Phi$ be a conformal iterated function system satisfying the open set condition. Suppose $\theta:\mathbb{N}\to \R_{\geq 0}$ satisfies $\theta(n)\asymp n^{-1/\dim_{H}(X)}$. Then $\mathcal{H}^{\dim_{H}(X)}$ almost every $x\in X$ is an element of $W(\theta,z).$
\end{cor}
\begin{proof}
We omit the details why for this choice of $\theta$ we have divergence in \eqref{Divergence condition}. Instead we give a brief argument explaining why \eqref{DufSch Divergence} holds. For any $\theta$ satisfying our assumptions we have
$$\sum_{n=1}^{Q}\theta(n)^{\dim_{H}(X)}\log \frac{1}{\theta(n)}\asymp \sum_{n=1}^{Q}\frac{\log n}{n}\asymp (\log Q)^2.$$ Similarly, we have
$$\Big(\sum_{n=1}^{Q}\theta(n)^{\dim_{H}(X)}\Big)^2\asymp \Big(\sum_{n=1}^{Q}\frac{1}{n}\Big)^2\asymp (\log Q)^2.$$ The final step in both of these equations can be seen to hold by approximating the summation with an integral. Combining these two equations we have \eqref{DufSch Divergence}. By Theorem \ref{Weak Duffin Schaeffer} we may conclude our result.
\end{proof}

In the statement of Theorem \ref{Weak Duffin Schaeffer} we introduced a new condition to replace the decreasing condition appearing in the definition of an approximation regular pair. Similarly, one can introduce a condition on the element $z$ which allows one to prove that $\mathcal{H}^{\dim_{H}(X)}$ almost every $x\in X$ is an element of $W(\theta,z)$ with no monotonicity assumptions on $\theta$. We postpone the statement of this condition until Section \ref{second theorems}. We will show in Section \ref{second theorems} that the set of points with this property has full Hausdorff dimension within $X.$ As a consequence of these results we have the following theorem.
\begin{thm}
\label{Duffin Schaeffer}
If $\Phi$ is a conformal iterated function system satisfying the open set condition, then there exists $Y\subset X$ satisfying $\dim_{H}(Y)=\dim_{H}(X),$ such that if $z\in Y$ and $\theta:\N\to \R_{\geq 0}$ satisfies \eqref{Divergence condition},
then $\mathcal{H}^{\dim_{H}(X)}$ almost every $x\in X$ is an element of $W(\theta,z).$
\end{thm}

\begin{remark}
The well informed reader might rightly ask whether the general framework introduced by Beresnevich, Dickinson, and Velani \cite{BDV} can be applied to give a proof of Theorem \ref{conformal theorem}. This general approach relies on the introduction of an appropriate weight function which satisfies certain properties. For self-similar sets with a uniform contraction ratio such a weight function can be shown to exist. However, for more general self-similar sets and self-conformal sets it is not clear whether such a function exists. Thus we do not apply their techniques. In any case, our proof of Theorem \ref{conformal theorem} is the starting point for the proofs of Theorem \ref{Weak Duffin Schaeffer} and Theorem \ref{Duffin Schaeffer}. Both of these theorems do not follow from the work done in \cite{BDV}.
\end{remark}
\begin{remark}
Under the assumptions of statement $2$ from Theorem \ref{conformal theorem}, it can be shown that $Diam(X_{I})^{\dim_{H}(X)}\asymp \mu(X_{I})$ for all $I\in\D^{*},$ where $\mu$ is a finite measure supported on $X$. This implies that
$$\sum_{I\in \D^{n}}Diam(X_{I})^{\dim_{H}(X)}\asymp 1$$ for all $n\in\N$. Thus condition \eqref{Divergence condition} is equivalent to $$\sum_{n=1}^{\infty}\theta(n)^{\dim_{H}(X)}=\infty.$$
\end{remark}
\begin{remark}
When $X$ is a self-similar set where each similarity has the same contraction ratio $r$, then \eqref{Divergence condition} can be rewritten as
\begin{equation}
\label{homogeneous divergence}
\sum_{n=1}^{\infty}\#\D^{n}(r^{n}\theta(n))^{\dim_{H}(X)}=\infty.
\end{equation} The rephrased divergence condition stated in \eqref{homogeneous divergence} is the same condition as that which appears in \cite{FishSim} and \cite{LSV}.
\end{remark}

The rest of this paper is arranged as follows. In Section \ref{examples section} we include some examples which demonstrate that restricting to approximating functions of the form appearing in the definition of an approximation regular pair is essential if one wants to expect the divergence of naturally occurring sums to provide any information about the size of $W(\Psi,z)$. In Section \ref{Main theorem} we prove Theorem \ref{conformal theorem}. We then prove Theorem \ref{Weak Duffin Schaeffer} and Theorem \ref{Duffin Schaeffer} in Section \ref{Next theorems}. In Section \ref{mass transference} we combine Theorem \ref{conformal theorem} with the mass transference principle of Beresnevich and Velani \cite{BerVel}. The mass transference principle will allow us to determine the Hausdorff dimension of the set $W(\Psi,z)$ for a certain class of $\Psi$ when we have convergence in \eqref{Divergence condition}. We apply this result to obtain a general criteria that allows one to deduce that $X$ contains exceptional elements, see Proposition \ref{height prop}. Exceptional here means well approximated in a way that maybe defined independently from $X$. As an application of this result, we obtain a general solution to a problem of Mahler, and prove that there are badly approximable numbers that are very well approximated by quadratic irrationals. In Section \ref{final section} we prove that if a conformal IFS contains an exact overlap then there are no approximation regular pairs. We also discuss the overlapping case and suggest some future directions.

\section{Examples}
\label{examples section}
In this section we include two examples which demonstrate that in the definition of an approximation regular IFS it is necessary and perhaps natural to restrict to approximating functions of the form $\Psi(I)=Diam(X_I)^{\dim_{S}(X)/\dim_{H}(X)}\theta(|I|)$. The first example demonstrates that any inhomogeneity, i.e. different rates of contraction, that maybe present within our IFS should be taken into consideration.
\begin{example}\label{scaling example}
Let $\Phi=\{\phi_{1},\phi_{2}\}$ where $\phi_{1}(x)=\frac{3x}{4}$ and $\phi_{2}(x)=\frac{x}{4}+\frac{3}{4}.$ This iterated function system satisfies the open set condition and the corresponding attractor is the unit interval $[0,1]$. For simplicity we take $z=0.$ The following argument can easily be adapted to an arbitrary $z\in[0,1]$.

Note that the image $\phi_{I}(0)$ is always the left endpoint of the interval $\phi_{I}([0,1]).$ Also note that the set of intervals $\{\phi_{I}([0,1])\}_{I\in \{1,2\}^{n}}$ always cover $[0,1]$, and if two of these intervals have a non-empty intersection then they intersect at a mutual endpoint.

Consider the limsup set
$$W:=\Big\{x\in [0,1]: x\in B(\phi_{I}(0), 2^{-|I|}) \textrm{ for i.m. }I\in\{1,2\}^{*}\Big\}.$$ Note that in the definition of $W$ the radii of the defining balls only depends upon the length of the word $I$. There are $2^{n}$ elements of $\{1,2\}^n$, therefore $\sum_{n=1}^{\infty}\sum_{I\in\{1,2\}^{n}}2^{-n}=\infty.$ Consequently, if an analogue of Khintchine's theorem held for IFS's where we didn't need to take into account the inhomogeneity of $\Phi$, we would expect $W$ to have full Lebesgue measure. We now show that this is not the case.

Given $x\in[0,1],$ we say that $(i_{n})_{n=1}^{\infty}\in\{1,2\}^{\infty}$ is a coding for $x$ if
$$x=\bigcap_{n=1}^{\infty}X_{i_{1},\ldots,i_{n}}.$$
Every $x\in [0,1]$ has a coding. This coding is unique for every $x\in X$ apart from a countable set of $x$ with precisely two codings. We let $\pi:\{1,2\}^{\N}\to [0,1]$ be the function which maps $(i_{n})$ to $\cap_{n=1}^{\infty}X_{i_{1},\ldots,i_{n}}.$ Let $\mathcal{P}$ be the Bernoulli measure on $\{1,2\}^{\N}$ which gives the digit $1$ mass $3/4$ and the digit $2$ mass $1/4$. The push forward of the measure $\mathcal{P}$ under the map $\pi$ is precisely the Lebesgue measure restricted to $[0,1]$. Applying the strong law of large numbers, we may conclude that for Lebesgue almost every $x\in [0,1]$ its coding $(i_{n})$ satisfies $$\lim_{m\to \infty} \frac{\#\{1\leq n\leq m: i_{n}=1\}}{m}\to \frac{3}{4}\textrm{ and } \lim_{m\to \infty} \frac{\#\{1\leq n\leq m: i_{n}=2\}}{m}\to \frac{1}{4}.$$ By the above, for any $\epsilon>0,$ we may pick a large $N\in \N$ such that the set $$A:=\Big\{x\in[0,1]:\frac{\#\{1\leq n\leq m: i_{n}=1\}}{m}\geq  \frac{5}{8}\textrm{ and } \frac{\#\{ 1\leq n\leq m:i_{n}=2\}}{m}\leq \frac{3}{8}\textrm{ for all } m\geq N\Big\}$$ has Lebesgue measure at least $1-\epsilon$.

We now obtain a bound for the cardinality of the set $$\Sigma_m:=\Big\{(i_{n})_{n=1}^{m}\in\{1,2\}^{m}:\frac{\#\{ 1\leq n\leq m: i_{n}=1\}}{m}\geq  \frac{5}{8}\Big\}.$$ This bound will rely on a result from probability theory. Suppose we have a sequence of independent random variables $X_{1}, X_{2} \ldots, X_{m}$. Let $$\overline{X}=\frac{1}{m}\sum_{n=1}^{m}X_{n}\textrm{ and } \mu=\frac{1}{m}\sum_{n=1}^{m} E(X_{n}).$$ The following bound is known as Hoeffding's inequality \cite{Hoe}.
\begin{lemma}\label{Hoefdding's inequality}
Suppose $0\leq X_{n}\leq 1$ for all $1\leq n \leq m,$ then for $0<t<1-\mu$
$$\textrm{Prob}(\overline{X}-\mu\geq t)\leq e^{-2mt^{2}}.$$
\end{lemma}With Lemma \ref{Hoefdding's inequality} in mind we let $\mathcal{P}'$ be the unbiased probability measure that gives digit $1$ mass $1/2$ and digit $2$ mass $1/2$. Then
$$\#\Sigma_m=2^{m}\cdot \mathcal{P}'\Big((i_{n})_{n=1}^{m}:\frac{\#\{ 1\leq n\leq m: i_{n}=1\}}{m}\geq  \frac{5}{8}\Big).$$ Applying Hoeffding's inequality we obtain
\begin{align}
\label{Sigma bounds}
\#\Sigma_{m}=2^m\cdot \mathcal{P}'\Big((i_{n})_{n=1}^{m}:\frac{\#\{ 1\leq n\leq m: i_{n}=1\}}{m}-\frac{1}{2}\geq  \frac{1}{8}\Big)\leq \Big(\frac{2}{e^{2/64}}\Big)^m.
\end{align}Equation \eqref{Sigma bounds} is the desired upper bound on the cardinality of $\Sigma_m$.

Returning to the interval $[0,1],$ we remark that if $x$ is contained in $B(\phi_I(0),2^{-|I|})$ for some $I\in \D^{*},$ then $x$ must be contained in either  $B(a_{i_1,\ldots ,i_{|I|}},2^{-|I|})$ or $B(b_{i_1,\ldots, i_{|I|}},2^{-|I|}),$ where $(i_n)$ is the coding for $x$ and $\phi_{i_1,\ldots , i_{|I|}}([0,1])=[a_{i_1,\ldots, i_{|I|}},b_{i_1,\ldots, i_{|I|}}]$. This is a consequence of how the intervals $\{\phi_{I}([0,1])\}_{I\in\{1,2\}^{m}}$ sit alongside one another in $[0,1],$ and because $x$ is always contained in the interval $\phi_{i_1,\ldots, i_{n}}([0,1]).$

We now use the preceding observation to obtain estimates on $\L(A\cap W)$. Here and throughout $\L$ denotes the Lebesgue measure. For any $N\in \N$ we have
\begin{align*}
\L(A\cap W)& \leq \L\Big(\bigcup_{m=N}^{\infty}\bigcup_{I\in \Sigma_{m}}(B(a_{i_1,\ldots ,i_{m}},2^{-m})\cup B(b_{i_1,\ldots, i_{m}},2^{-m})\Big)\\
&\leq \sum_{m=N}^{\infty} \sum_{I\in \Sigma_m} \L(B(a_{i_1,\ldots, i_m},2^{-m})) + \L(B(b_{i_1,\ldots, i_m},2^{-m}))\\
&= \sum_{m=N}^{\infty}\#\Sigma_m\cdot 4\cdot 2^{-m} \\
&\leq \sum_{m=N}^{\infty}4\cdot\Big(\frac{1}{e^{2/64}}\Big)^m\\
&<\infty.
\end{align*}In the penultimate inequality we used \eqref{Sigma bounds}. Since we have convergence above, we can choose $N\in \N$ such that $\sum_{m=N}^{\infty}4\cdot e^{-2m/64}$ is arbitrarily small. Therefore $\L(A\cap W)=0.$ Since $\L(A)>1-\epsilon$ we have $\L(W)<\epsilon.$ Since $\epsilon$ is arbitrary we may conclude that $\L(W)=0$.
\end{example}

Our second example demonstrates that for a reasonable analogue of Khintchine's theorem to hold for IFS's, it is necessary that an approximation function gives weight to all words and is not concentrated on a subset of $\D^{*}.$

\begin{example}
\label{uniform example}
Let $\Phi=\{\phi_{1},\phi_{2}\}$ where $\phi_{1}(x)=\frac{x}{2}$ and $\phi_{2}(x)=\frac{x}{2}+\frac{1}{2}.$ As in Example \ref{scaling example} our attractor is the interval $[0,1]$ and $\Phi$ satisfies the open set condition. We fix $z=0$ and a word $J=(j_1,\ldots,j_m)\in\{1,2\}^{m}$. Let
\[
 \Psi(I) =
  \begin{cases}
   2^{-|I|} & \text{if } I \textrm{ doesn't begin with } J \\
   2^{-|I|}|I|^{-2}       & \text{if } I \textrm{ begins with } J.
  \end{cases}
\]Then $\Psi$ satisfies
\begin{equation}
\label{convergence and divergencea}
\sum_{\substack {I\in \D^{*}\\ I \textrm{ begins with } J}}\Psi(I)<\infty
\end{equation}
and
\begin{equation}
\label{convergence and divergence}
\sum_{n=1}^{\infty}\sum_{I\in \D^{n}}\Psi(I)=\infty.
\end{equation}If an analogue of Khintchine's theorem held for IFS's where $\Psi$ did not have to distribute weight evenly amongst the elements of $\D^{*},$ then \eqref{convergence and divergence} would suggest $W(\Psi,0)$ has full Lebesgue measure. However, using the Borel Cantelli lemma we can show that \eqref{convergence and divergencea} implies
$$\mathcal{L}\Big(\Big[\sum_{i=1}^{m}\frac{j_{i}}{2^{i}},\sum_{i=1}^{m}\frac{j_{i}}{2^{i}}+\frac{1}{2^{i}}\Big]\cap W(\Psi,0)\Big)=0.$$ Thus $W(\Psi,0)$ does not have full measure despite \eqref{convergence and divergence} being satisfied.
\end{example}
Bearing Example \ref{scaling example} and Example \ref{uniform example} in mind, we believe that the class of approximating functions we restrict to in the definition of an approximation regular pair is not so restrictive, and is in fact a natural class of approximating functions to study.

\section{Proof of Theorem \ref{conformal theorem}}
\label{Main theorem}
\subsection{Preliminaries}
We start this section by recalling the definition of Hausdorff measure and Hausdorff dimension. Let $E\subset \R^{d},$ $s\geq 0,$ and $\rho>0.$ We let $$\mathcal{H}^{s}_{\rho}(E):=\inf\Big\{\sum_{n=1}^{\infty}Diam(U_n)^s: \{U_n\}_{n=1}^{\infty}\textrm{ is a }\rho\textrm{-cover for }E\Big\}.$$ In the above the infimum is taken over all $\rho$-covers of $E$. The limit $\lim_{\rho\to 0}\mathcal{H}^{s}_{\rho}(E):=\mathcal{H}^{s}(E)$ exists and we call this limit the \emph{$s$-dimensional Hausdorff measure} of $E$. It is a straightforward exercise to show that for any $E\subset \R^{d}$ the following equality holds $$\inf\{s:\mathcal{H}^{s}(E)=0\}=\sup\{s:\mathcal{H}^{s}(E)=\infty\}.$$ We call this coinciding value the \emph{Hausdorff dimension} of $E$ and denote it by $\dim_{H}(E)$.

Suppose $E$ has Hausdorff dimension $s.$ We say that $E$ is \emph{Ahlfors regular} if
\begin{equation}
\label{Ahlfors regular}\mathcal{H}^s(E\cap B(x,r))\asymp r^s,
\end{equation} for all $x\in E$ and $0<r<Diam(E).$ Importantly, under the open set condition of statement $2$ from Theorem \ref{conformal theorem}, the attractor $X$ will always be Ahlfors regular.

In our proofs we will require the notion of a coding. This is the natural generalisation of what appeared in Example \ref{scaling example}. Given an IFS $\Phi$ and $x\in X,$ then there exists a sequence $(i_n)\in \D^{\N}$ such that $$x=\bigcap_{n=1}^{\infty}X_{i_{1},\ldots,i_{n}}.$$ We call such a sequence a \emph{coding of $x$}. The coding of $x$ is not necessarily unique. An $x$ may well have a continuum of codings. As a final observation, we remark that if $(i_n)$ is a coding for $x,$ then $(j_1,\ldots, j_m, i_1,i_2,\ldots)$ is a coding for $\phi_{J}(x)$ where $J=(j_1,\ldots, j_m).$

\subsection{Statement 1}
The proof of statement $1$ from Theorem \ref{conformal theorem} is standard but we include it for completeness.

\begin{proof}[Proof of statement $1$]
Let $\Phi$ be an IFS with attractor $X$ and let $z$ be an arbitrary element of $X$. Let $\Psi:\D^{*}\to \mathbb{R}_{\geq 0}$ be an approximating function satisfying
\begin{equation}
\label{Proof convergence}
\sum_{n=1}^{\infty} \sum_{I\in\D^{n}} \Psi(I)^{\dim_{H}(X)}<\infty.
\end{equation} Fix $\rho>0$. Let $M\in\mathbb{N}$ be sufficiently large that
$2\Psi(I)<\rho$ for all $I\in \D^n$ with $n\geq M$. Such an $M$ exists because of \eqref{Proof convergence}. Therefore $$\{B(\phi_{I}(z),\Psi(I))\}_{\substack{I\in \D^n\\ n\geq M}}$$ is a $\rho$ cover of $W(\Psi,z)$. For any $\epsilon>0,$ one can assume that $M$ is also sufficiently large that
\begin{equation*}
\sum_{n=M}^{\infty} \sum_{I\in\D^{n}} \Psi(I)^{\dim_{H}(X)}<\epsilon.
\end{equation*}
This is a consequence of \eqref{Proof convergence}. Therefore
\begin{align*}
\mathcal{H}^{\dim_{H}(X)}_{\rho}(W(\Psi,z))& \leq \sum_{n=M}^{\infty}\sum_{I\in \D^{n}} (2\Psi(I))^{\dim_{H}(X)}\\
&\leq 2^{\dim_{H}(X)}\epsilon.
\end{align*}Since $\epsilon$ is arbitrary we must have $\mathcal{H}^{\dim_{H}(X)}_{\rho}(W(\Psi,z))=0.$ Since $\rho$ was arbitrary we have $\mathcal{H}^{\dim_{H}(X)}(W(\Psi,z))=0.$

\end{proof}

\subsection{Statement 2 }
The proof of statement $2$ from Theorem \ref{conformal theorem} follows a similar framework to the proof of Theorem $2$ from \cite{LSV}. In particular we make use of the following two lemmas.

\begin{lemma}
\label{Density lemma}
Let $X$ be a compact set in $\R^{d}$ and let $\mu$ be a finite doubling measure on X such
that any open set is $\mu$ measurable. Let $E$ be a Borel subset of $X$. Assume that there are
constants $r_0,c > 0$ such that for any ball $B$ with radius less than $r_0$ and centre in $X$ we have
$$\mu(E \cap B) > c \mu(B).$$
Then $\mu(X \setminus E) = 0.$
\end{lemma}For a proof of Lemma \ref{Density lemma} see \cite[\S 8]{BDV}. Note that a measure $\mu$ supported on a compact set $X$ is \emph{doubling} if there exists a constant $C>1$ such that for any $x\in X$ and $r>0$ we have $$\mu(B(x,2r))\leq C \mu(B(x,r)).$$ Clearly if $X$ is Ahlfors regular then the restriction of $\mathcal{H}^{\dim_{H}(X)}$ to $X$ is a doubling measure.
\begin{lemma}
\label{Erdos lemma}
Let $X$ be a compact set in $\R^{d}$ and let $\mu$ be a finite measure on $X$. Also, let $E_n$ be
a sequence of $\mu$-measurable sets such that $\sum_{n=1}^{\infty}\mu(E_n)=\infty.$ Then
$$\mu(\limsup_{n\to\infty} E_{n})\geq \limsup_{Q\to\infty}\frac{(\sum_{n=1}^{Q}\mu(E_{n}))^{2}}{\sum_{n,m=1}^{Q}\mu(E_{n}\cap E_m)}.$$
\end{lemma}For a proof of Lemma \ref{Erdos lemma} see \cite[Lemma 5]{Spr}.

In the proofs of each of our theorems we will need the following properties of self-conformal sets satisfying the open set condition. Let $\mu:=\mathcal{H}^{\dim_{H}(X)}|_{X}$ be the $\dim_{H}(X)$-dimensional Hausdorff measure restricted to $X,$ then:

\begin{itemize}
\item For any $n\in\mathbb{N}$ and $I,J\in D^{n}$ such that $I\neq J$ we have
\begin{equation}
\label{measure zero intersection}
\mu(X_{I}\cap X_{J})=0.
\end{equation}
\item For any $I,J\in \D^{*}$
\begin{equation}
\label{Weak Bernoulli measure}
\mu(X_{IJ})\asymp \mu(X_{I})\mu(X_{J}).
\end{equation}
  \item  For any $I,J\in \D^{*}$
\begin{equation}
\label{Weak Bernoulli diameter}
Diam(X_{IJ})\asymp Diam(X_{I})Diam(X_{J}).
\end{equation}
  \item For any $I\in \D^{*}$
\begin{equation}
\label{Measure and diameter}
\mu(X_{I})\asymp Diam(X_{I})^{\dim_{H}(X)}.
\end{equation}
\item There exists $\gamma\in(0,1)$ such that for any $I\in \D^{*}$
\begin{equation}
\label{measure decay}
\mu(X_{I})= \mathcal{O}(\gamma^{|I|}).
\end{equation}
\item Let $x\in X$ and $(i_n)\in \D^{\N}$ be a coding of $x$. For any $0<r<Diam(X)$ there exists $N\in\mathbb{N}$ such that
\begin{equation}
\label{cylinder approx}
X_{i_{1},\ldots,i_{N}}\subset B(x,r) \textrm{ and } Diam(X_{i_{1},\ldots,i_{N}})\asymp r.
\end{equation}

\end{itemize}
In the above we have denoted the concatenation of two words $I$ and $J$ by $IJ$. Property \eqref{measure zero intersection} follows from \cite[Theorem 3.7]{Kae}. For a proof of the remaining properties and for a proof that $X$ is Ahlfors regular see \cite{Fal} and \cite{Rue}. Properties \eqref{Weak Bernoulli measure}, \eqref{Measure and diameter}, and \eqref{measure decay} are essentially a consequence of the fact that $\mu$ is equivalent to a suitably defined Gibbs measure for a particular H\"{o}lder continuous potential. Property \eqref{Weak Bernoulli diameter} is a consequence of the differentials being H\"{o}lder continuous. The proof of \eqref{cylinder approx} is standard.

\subsubsection{Proof of statement 2}
We are now in a position to prove statement $2$ from Theorem \ref{conformal theorem}. We start by fixing $z\in X$ and let $(z_i)\in \D^{\N}$ be a coding of $z$. Let $\theta:\N\to \mathbb{R}$ be a decreasing function satisfying \eqref{Divergence condition}. We now fix a ball $B(y,r)$ where $y\in X$ and $0<r<Diam(X)$. We will show that
\begin{equation}
\label{limsup density}
\mu(W(\theta,z) \cap B(y,r)) > c \mu(B(y,r)),
\end{equation}for some constant $c>0$ that does not depend on our choice of ball. Applying Lemma \ref{Density lemma} will then allow us to conclude our result.

The set $W(\theta,z)$ is defined to be a limsup set of a sequence of balls. It will be computationally easier to replace these balls with cylinder sets. For each $I\in \D^{*}$ we let $X_{I,\theta}$ be a cylinder set that satisfies the following two properties:
\begin{align}
X_{I,\theta}&\subseteq B(\phi_{I}(z),Diam(X_{I})\theta(|I|))\label{align1}\\
\mu(X_{I,\theta})&\asymp (Diam(X_{I})\theta(|I|))^{\dim_{H}(X)}\label{align2}.
\end{align}Such a cylinder set exists by properties \eqref{Measure and diameter} and \eqref{cylinder approx}. Without loss of generality we may assume that $$X_{I,\theta}=X_{I(z_1,\ldots, z_{n(I,\theta)})}$$ for some $n(I,\theta)\in \N.$ We emphasise here that for any $I\in \D^n$ we have that $I$ is a prefix of $I(z_1,\ldots, z_{n(I,\theta)}),$  $|I(z_1,\ldots, z_{n(I,\theta)})|=n+n(I,\theta),$ and $X_{I,\theta}\subset X_{I}$.  Importantly, by \eqref{measure zero intersection} we have $\mu(X_{I,\theta}\cap X_{J,\theta})=0$ for $I,J\in\D^{n}$ such that $I\neq J.$ We now replace $W(\theta,z)$ with the following limsup set that is defined using cylinder sets instead of balls. Let $$E(\theta,z):=\Big\{x\in X: x\in X_{I,\theta} \textrm{ for }i.m.\, I\in \D^{*}\Big\}$$ By \eqref{align1} we have $E(\theta,z)\subseteq W(\theta,z)$. So to prove \eqref{limsup density} it suffices to prove
\begin{equation}
\label{limsup density2}
\mu(E(\theta,z) \cap B(y,r)) > c \mu(B(y,r)).
\end{equation}

To our ball $B(y,r)$ we associate the cylinder $X_{(y_{1},\ldots,y_{n(r)})},$ where $(y_i)\in \D^{\N}$ is a coding for $y$ and $X_{(y_{1},\ldots,y_{n(r)})}$ satisfies \eqref{cylinder approx}. To each $n\geq n(r)$ we associate the sets
$$\mathcal{E}_{n}:=\Big\{X_{I,\theta}: I\in\D^{n} \textrm{ and } (i_{1},\ldots,i_{n(r)})=(y_{1},\ldots,y_{n(r)})\Big\} $$ and $$E_{n}:=\bigcup_{X_{I,\theta}\in \mathcal{E}_{n}} X_{I,\theta}.$$ Then $$\limsup E_n\subset E(\theta,z)\cap X_{(y_{1},\ldots,y_{n(r)})}\subset E(\theta,z) \cap B(y,r).$$ Therefore to prove \eqref{limsup density2} it suffices to show that
\begin{equation}
\label{limsup density4}
\mu(\limsup E_n) > c \mu(B(y,r)).
\end{equation}It is a consequence of \eqref{Measure and diameter}, \eqref{cylinder approx}, and the fact that $\mu$ is Ahlfors regular that $$\mu(X_{(y_{1},\ldots,y_{n(r)})})\asymp \mu(B(y,r)).$$ Therefore to prove \eqref{limsup density4} it suffices to prove
\begin{equation}
\label{limsup density3}
\mu(\limsup E_n) > c \mu(X_{(y_{1},\ldots,y_{n(r)})}).
\end{equation}We now focus our attention on proving \eqref{limsup density3}. To prove \eqref{limsup density3} we will apply Lemma \ref{Erdos lemma} to the sequence of sets $E_n$. To successfully apply this lemma there are two steps, we must first demonstrate that $\sum \mu(E_n)=\infty$, we then obtain estimates for the measure of $E_{n}\cap E_{m}$ for $n\neq m$. We split these steps into the following two propositions.
\begin{prop}
\label{Divergence prop}
$\sum_{n=n(r)}^{\infty} \mu(E_n)=\infty.$
\end{prop}
\begin{proof}
Recalling \eqref{Divergence condition}, our $\theta$ must satisfy
\begin{equation}
\label{kk}
\sum_{n=n(r)}^{\infty} \sum_{I\in\D^{n}} (Diam(X_I)\theta(n))^{\dim_{H}(X)}=\infty.
 \end{equation}Each $X_{I,\theta}$ satisfies $\mu(X_{I,\theta})\asymp (Diam(X_I)\theta(|I|))^{\dim_{H}(X)}$ by \eqref{align2}. Therefore
\begin{equation}
\label{hh}
\sum_{n=n(r)}^{M} \sum_{I\in\D^{n}} (Diam(X_I)\theta(n))^{\dim_{H}(X)}\asymp \sum_{n=n(r)}^{M} \sum_{I\in\D^{n}}\mu( X_{I,\theta})
 \end{equation}for any $M\geq n(r)$. It is a consequence of \eqref{Weak Bernoulli measure} that for any $J\in \D^{n(r)}$ we have
\begin{equation}
\label{simultaneous divergence}
\sum_{n=n(r)}^{M} \sum_{\substack{I\in\D^{n}\\ (i_1,\ldots,i_{n(r)})=(y_1,\ldots,y_{n(r)})}}\mu( X_{I,\theta})\asymp \sum_{n=n(r)}^{M} \sum_{\substack{I\in\D^{n}\\ (i_1,\ldots,i_{n(r)})=J}}\mu( X_{I,\theta}).
\end{equation} In \eqref{simultaneous divergence} the implied constants may depend on $n(r)$. By \eqref{simultaneous divergence} we see that
\begin{equation}
\label{simultaneous divergence2}
\sum_{n=n(r)}^{\infty} \sum_{\substack{I\in\D^{n}\\ (i_1,\ldots,i_{n(r)})=(y_1,\ldots,y_{n(r))}}}\mu( X_{I,\theta})=\infty \textrm{ if and only if }\sum_{n=n(r)}^{\infty} \sum_{\substack{I\in\D^{n}\\ (i_1,\ldots,i_{n(r)})=J}}\mu( X_{I,\theta})=\infty.
 \end{equation} By \eqref{kk} and \eqref{hh} we must have
 $$\sum_{n=n(r)}^{\infty} \sum_{\substack{I\in\D^{n}\\ (i_1,\ldots,i_{n(r)})=J}}\mu( X_{I,\theta})=\infty$$ for at least one $J\in\D^{n(r)}$. Therefore by \eqref{simultaneous divergence2} we may conclude that $$\sum_{n=n(r)}^{\infty} \sum_{\substack{I\in\D^{n}\\ (i_1,\ldots,i_{n(r)})=(y_1,\ldots,y_{n(r)})}}\mu( X_{I,\theta})=\infty.$$ Since distinct elements of $\mathcal{E}_{n}$ intersect in a set of measure zero we have $$\sum_{n=n(r)}^{\infty} \sum_{\substack{I\in\D^{n}\\ (i_1,\ldots,i_{n(r)})=(y_1,\ldots,y_{n(r)})}}\mu( X_{I,\theta})=\sum_{n=n(r)}^{\infty} \mu(E_n).$$ Therefore $\sum_{n=n(r)}^{\infty} \mu(E_n)=\infty$ as required.
\end{proof}

\begin{prop}Let $\theta:\mathbb{N}\to \R_{\geq 0}$ be decreasing. Then
\label{independence prop}
$$\sum_{n,m=n(r)}^{Q}\mu(E_{n}\cap E_{m})=\O\Big(\mu(X_{(y_1,\ldots,y_{n(r)})})\Big(\sum_{n=n(r)}^{Q}\theta(n)^{\dim_{H}(X)}+\Big(\sum_{n=n(r)}^{Q}\theta(n)^{\dim_{H}(X)}\Big)^2\Big)\Big).$$
\end{prop}
\begin{proof}
Let $I\in \D^{n}$ and $X_{I,\theta}\in \mathcal{E}_{n}.$ When calculating $\mu(X_{I,\theta}\cap E_{m})$ there are two cases that naturally arise, when $m> n+n(I,\theta)$ and when $n<m\leq n+n(I,\theta).$ Let us start with the case where $n<m\leq n+n(I,\theta).$ Since the rank of the cylinder $X_{I,\theta}$ is at least $m,$ there exists at most one $X_{J,\theta}\in \mathcal{E}_{m}$ such that $X_{I,\theta}\cap X_{J,\theta}\neq \emptyset.$ Moreover this $J$ must be of the form $J=(i_1,\ldots, i_n,z_1,\ldots,z_{m-n})$. These observations give rise to the following useful bound on $\mu(X_{I,\theta}\cap E_{m})$ for $n<m\leq n+n(I,\theta)$:
\begin{align*}
\mu(X_{I,\theta}\cap E_{m})&= \mu(X_{I,\theta}\cap X_{J,\theta})\\
&\leq \mu(X_{J,\theta})\\
&\asymp (Diam(X_{J})\theta(m))^{\dim_{H}(X)}\, &&(\textrm{By }\eqref{align2})\\
&\asymp (Diam(X_{I})Diam(X_{(z_{1},\ldots,z_{m-n})})\theta(m))^{\dim_{H}(X)}\, &&(\textrm{By }\eqref{Weak Bernoulli diameter})\\
&\asymp \mu(X_{I})\mu(X_{(z_{1},\ldots,z_{m-n})})\theta(m)^{\dim_{H}(X)}\, &&(\textrm{By }\eqref{Measure and diameter})\\
&\leq \mu(X_{I})\mu(X_{(z_{1},\ldots,z_{m-n})})\theta(n)^{\dim_{H}(X)}\, && (\textrm{Because }\theta \textrm{ is decreasing})\\
&= \mathcal{O}( \mu(X_{I})\theta(n)^{\dim_{H}(X)}\gamma^{m-n})\, && (\textrm{By }\eqref{measure decay}).
\end{align*}We have shown that if $n<m\leq n+n(I,\theta)$ then
\begin{equation}
\label{intersection bound 1}
\mu(X_{I,\theta}\cap E_{m})= \mathcal{O}( \mu(X_{I})\theta(n)^{\dim_{H}(X)}\gamma^{m-n}).
\end{equation}Now let us consider the case where $m> n+n(I,\theta).$ In this case
$$X_{I,\theta}\cap E_{m}=\Big\{X_{J,\theta}:J\in\D^{m}\textrm{ and } (j_{1},\ldots,j_{n+n(I,\theta)})=(i_1,\ldots,i_n,z_1,\ldots,z_{n(I,\theta)})\Big\}.$$Thus we obtain the following bounds
\begin{align*}
\mu(X_{I,\theta}\cap E_{m})&=\sum_{\substack{J\in \D^{m}\\ (j_{1},\ldots,j_{n+n(I,\theta)})=I(z_1,\ldots,z_{n(I,\theta)})}}\mu(X_{J,\theta})\\
&=\sum_{K\in\D^{m-n-n(\theta,I)}}\mu(X_{I(z_{1},\ldots,z_{n(I,\theta)})K,\theta})\\
&\asymp \sum_{K\in\D^{m-n-n(\theta,I)}} (Diam(X_{I(z_{1},\ldots,z_{n(I,\theta)})K})\theta(m))^{\dim_{H}(X)} &&(\textrm{By }\eqref{align2})\\
&\asymp (Diam(X_{I,\theta}))\theta(m))^{\dim_{H}(X)}\sum_{K\in\D^{m-n-n(\theta,I)}} Diam(X_{K})^{\dim_{H}(X)} && (\textrm{By }\eqref{Weak Bernoulli diameter})\\
&\asymp \mu(X_{I,\theta})\theta(m)^{\dim_{H}(X)}\sum_{K\in\D^{m-n-n(\theta,I)}} \mu(X_{K}) && (\textrm{By }\eqref{Measure and diameter})\\
&\asymp (Diam(X_{I})\theta(n)\theta(m))^{\dim_{H}(X)} && (\textrm{By }\eqref{align2})\\
&\asymp \mu(X_{I})\theta(n)^{\dim_{H}(X)}\theta(m)^{\dim_{H}(X)} &&(\textrm{By }\eqref{Measure and diameter}).
\end{align*}We have shown that if $m> n+n(I,\theta)$ then
\begin{equation}
\label{intersection bound 2}
\mu(X_{I,\theta}\cap E_{m})\asymp \mu(X_{I})\theta(n)^{\dim_{H}(X)}\theta(m)^{\dim_{H}(X)}.
\end{equation}Combining \eqref{intersection bound 1} and \eqref{intersection bound 2} we obtain the bound
\begin{equation}
\label{intersection bound}
\mu(X_{I,\theta}\cap E_{m})\leq \mathcal{O}\Big(\mu(X_{I})\theta(n)^{\dim_{H}(X)}\gamma^{m-n}+\mu(X_{I})\theta(n)^{\dim_{H}(X)}\theta(m)^{\dim_{H}(X)}\Big).
\end{equation} Importantly this bound holds for all $m> n$.


This implies the following
\begin{align}
\sum_{n,m=n(r)}^{Q}\mu(E_{n}\cap E_{m})&=2\sum_{n=n(r)}^{Q}\mu(E_{n})+2\sum_{n=n(r)}^{Q-1}\sum_{m=n+1}^{Q}\mu(E_{n}\cap E_{m})\nonumber\\
&=2\sum_{n=n(r)}^{Q}\mu(E_{n})+2\sum_{n=n(r)}^{Q-1}\sum_{\substack{I\in\D^{n} \nonumber\\ (i_{1},\ldots,i_{n(r)})=(y_{1},\ldots,y_{n(r)})}}\sum_{m=n+1}^{Q}\mu(X_{I,\theta}\cap E_{m})\nonumber\\
&= 2\sum_{n=n(r)}^{Q}\mu(E_{n})+\mathcal{O}\Big(\sum_{n=n(r)}^{Q-1}\sum_{\substack{I\in\D^{n} \nonumber\\ (i_{1},\ldots,i_{n(r)})=(y_{1},\ldots,y_{n(r)})}}\sum_{m=n+1}^{Q}\mu(X_{I})\theta(n)^{\dim_{H}(X)}\gamma^{m-n}\nonumber\Big)\\
&+ \O\Big(\sum_{n=n(r)}^{Q-1}\sum_{\substack{I\in\D^{n}\\ (i_{1},\ldots,i_{n(r)})=(y_{1},\ldots,y_{n(r)})}}\sum_{m=n+1}^{Q} \mu(X_{I})\theta(n)^{\dim_{H}(X)}\theta(m)^{\dim_{H}(X)}\Big). \label{three parts}
\end{align}
We now consider each of the three terms appearing in \eqref{three parts} individually. For the first term we have
\begin{align*}
\sum_{n=n(r)}^{Q}\mu(E_{n})&=\sum_{n=n(r)}^{Q}\sum_{I\in \D^{n-n(r)}}\mu(X_{(y_1,\ldots,y_{n(r)})I,\theta}) \\
&\asymp \sum_{n=n(r)}^{Q}\sum_{I\in \D^{n-n(r)}} (Diam(X_{(y_1,\ldots,y_{n(r)})I})\theta(n))^{\dim_{H}(X)} && (\textrm{By }\eqref{align2})\\
&\asymp \sum_{n=n(r)}^{Q}\sum_{I\in \D^{n-n(r)}}(Diam(X_{(y_1,\ldots,y_{n(r)})})Diam(X_{I})\theta(n))^{\dim_{H}(X)} &&(\textrm{By }\eqref{Weak Bernoulli diameter})\\
&\asymp \mu(X_{(y_1,\ldots,y_{n(r)})})\sum_{n=n(r)}^{Q}\theta(n)^{\dim_{H}(X)} \sum_{I\in \D^{n-n(r)}}\mu(X_{I}) &&(\textrm{By }\eqref{Measure and diameter})\\
&\asymp \mu(X_{(y_1,\ldots,y_{n(r)})})\sum_{n=n(r)}^{Q}\theta(n)^{\dim_{H}(X)}.
\end{align*}Thus we have shown that
\begin{equation}
\label{final bound 1}
\sum_{n=n(r)}^{Q}\mu(E_{n})\asymp \mu(X_{(y_1,\ldots,y_{n(r)})})\sum_{n=n(r)}^{Q}\theta(n)^{\dim_{H}(X)}.
\end{equation}
We now focus on the second term in \eqref{three parts}:
\begin{align}
\label{bound2}
&\sum_{n=n(r)}^{Q-1}\sum_{\substack{I\in\D^{n} \\ (i_{1},\ldots,i_{n(r)})=(y_{1},\ldots,y_{n(r)})}}\sum_{m=n+1}^{Q}\mu(X_{I})\theta(n)^{\dim_{H}(X)}\gamma^{m-n}\nonumber\\
&\asymp\mu(X_{(y_1,\ldots,y_{n(r)})})\sum_{n=n(r)}^{Q-1}\sum_{J\in\D^{n-n(r)}}\sum_{m=n+1}^{Q}\mu(X_{J})\theta(n)^{\dim_{H}(X)}\gamma^{m-n}&& (\textrm{By }\eqref{Weak Bernoulli measure})\nonumber \\
&=\mu(X_{(y_1,\ldots,y_{n(r)})})\sum_{n=n(r)}^{Q-1}\theta(n)^{\dim_{H}(X)}\sum_{J\in\D^{n-n(r)}}\mu(X_{J})\sum_{m=n+1}^{Q}\gamma^{m-n}\nonumber\\
&= \O\Big( \mu(X_{(y_1,\ldots,y_{n(r)})})\sum_{n=n(r)}^{Q-1}\theta(n)^{\dim_{H}(X)}\Big).
\end{align}In the last line above we used the fact that $\gamma\in(0,1)$ to conclude that $\sum_{m=n+1}^{Q}\gamma^{m-n}$ can be bounded above by a constant independent of $n$ and $Q$. We now turn our attention to the third term in \eqref{three parts}:
\begin{align}
\label{bound3}
&\sum_{n=n(r)}^{Q-1}\sum_{\substack{I\in\D^{n} \nonumber\\ (i_{1},\ldots,i_{n(r)})=(y_{1},\ldots,y_{n(r)})}}\sum_{m=n+1}^{Q}\mu(X_{I})\theta(n)^{\dim_{H}(X)}\theta(m)^{\dim_{H}(X)}\nonumber\\
&\asymp \mu(X_{(y_1,\ldots,y_{n(r)})})\sum_{n=n(r)}^{Q-1}\sum_{J\in\D^{n-n(r)}}\sum_{m=n+1}^{Q}\mu(X_{J})\theta(n)^{\dim_{H}(X)}\theta(m)^{\dim_{H}(X)}&& (\textrm{By }\eqref{Weak Bernoulli measure})\nonumber \\
& =\mu(X_{(y_1,\ldots,y_{n(r)})})\sum_{n=n(r)}^{Q-1}\theta(n)^{\dim_{H}(X)}\sum_{J\in\D^{n-n(r)}}\mu(X_{J})\sum_{m=n+1}^{Q}\theta(m)^{\dim_{H}(X)}\nonumber\\
& =\mu(X_{(y_1,\ldots,y_{n(r)})})\sum_{n=n(r)}^{Q-1}\theta(n)^{\dim_{H}(X)}\sum_{m=n+1}^{Q}\theta(m)^{\dim_{H}(X)}\nonumber\\
& = \O\Big(\mu(X_{(y_1,\ldots,y_{n(r)})})\Big(\sum_{n=n(r)}^{Q}\theta(n)^{\dim_{H}(X)}\Big)^2\Big).
\end{align}Substituting \eqref{final bound 1}, \eqref{bound2}, and \eqref{bound3} into \eqref{three parts} we obtain $$\sum_{n,m=n(r)}^{Q}\mu(E_{n}\cap E_{m})= \O\Big(\mu(X_{(y_1,\ldots,y_{n(r)})})\Big(\sum_{n=n(r)}^{Q}\theta(n)^{\dim_{H}(X)}+\Big(\sum_{n=n(r)}^{Q}\theta(n)^{\dim_{H}(X)}\Big)^2\Big)\Big)$$ as required.

\end{proof}
\begin{proof}[Proof of statement $2$ from Theorem \ref{conformal theorem}]
By Proposition \ref{independence prop} and \eqref{final bound 1} there exists a $C>0$ such that 
$$\frac{(\sum_{n=n(r)}^{Q}\mu(E_{n}))^{2}}{\sum_{n,m=n(r)}^{Q}\mu(E_{n}\cap E_m)}$$ can be bounded below by
\begin{equation}
\label{limsup equation}
\frac{\Big(\mu(X_{(y_{1},\ldots,y_{n(r)})})\sum_{n=n(r)}^{Q}\theta(n)^{\dim_{H}(X)}\Big)^2}{C\mu(X_{(y_1,\ldots,y_{n(r)})})\Big(\sum_{n=n(r)}^{Q}\theta(n)^{\dim_{H}(X)}+\Big(\sum_{n=n(r)}^{Q}\theta(n)^{\dim_{H}(X)}\Big)^2\Big)}.
\end{equation}
It is a consequence of our approximating function $\theta$ satisfying the divergence condition \eqref{Divergence condition}, that for $Q$ sufficiently large $$\sum_{n=n(r)}^{Q}\theta(n)^{\dim_{H}(X)}>1,$$ and therefore $$\sum_{n=n(r)}^{Q}\theta(n)^{\dim_{H}(X)}<\Big(\sum_{n=n(r)}^{Q}\theta(n)^{\dim_{H}(X)}\Big)^2.$$ Taking limits in \eqref{limsup equation} we obtain
\begin{equation}
\label{limsup bound}
\limsup_{Q\to\infty}\frac{(\sum_{n=n(r)}^{Q}\mu(E_{n}))^{2}}{\sum_{n,m=n(r)}^{Q}\mu(E_{n}\cap E_m)}\geq \frac{\mu(X_{(y_{1},\ldots,y_{n(r)})})}{2C}.
\end{equation}
By Proposition \ref{Divergence prop} we may apply Lemma \ref{Erdos lemma}. Applying Lemma \ref{Erdos lemma} and \eqref{limsup bound} we obtain $$\mu(\limsup E_n)\geq \frac{\mu(X_{(y_{1},\ldots,y_{n(r)})})}{2C}.$$ Thus we may conclude that \eqref{limsup density3} holds and we have completed our proof.
\end{proof}

\section{Proofs of Theorem \ref{Weak Duffin Schaeffer} and Theorem \ref{Duffin Schaeffer}}
\label{second theorems}
The critical part of the proof of Theorem \ref{conformal theorem} was obtaining estimates for $\mu(X_{I,\theta}\cap E_m)$ when $n<m\leq n+n(I,\theta).$ Indeed this was the only point in our proof where the decreasing assumption on $\theta$ was used. To weaken the monotonicity assumptions used in the proof of Theorem \ref{conformal theorem}, we need new methods to obtain bounds on $\mu(X_{I,\theta}\cap E_m)$ for $n<m\leq n+n(I,\theta).$ The hypothesis appearing in Theorem \ref{Weak Duffin Schaeffer} and the condition appearing in Theorem \ref{Duffin Schaeffer} provide different methods for bounding $\mu(X_{I,\theta}\cap E_m)$ for $n<m\leq n+n(I,\theta).$
\label{Next theorems}
\subsection{Proof of Theorem \ref{Weak Duffin Schaeffer}}
Let $z\in X$, $y\in X$ and $r>0$ all be as in the proof of Theorem \ref{conformal theorem}. As in the proof of Theorem \ref{conformal theorem}, to prove Theorem \ref{Weak Duffin Schaeffer}, it suffices to show that \eqref{limsup density3} holds. The first step in proving Theorem \ref{Weak Duffin Schaeffer} is to prove the following more general analogue of Proposition \ref{independence prop}.


\begin{prop}
\label{DufSch prop}
For any $\theta:\mathbb{N}\to (0,1/2)$ we have
\begin{align}
\sum_{n,m=n(r)}^{Q}\mu(E_{n}\cap E_{m})=\O\Big(\mu(X_{(y_{1},\ldots,y_{n(r)})})\Big(&\sum_{n=n(r)}^{Q}\theta(n)^{\dim_{H}(X)}+\sum_{n=n(r)}^{Q}\theta(n)^{\dim_{H}(X)} \log \Big(\frac{1}{\theta(n)}\Big) \nonumber\\
&+ \Big(\sum_{n=n(r)}^{Q}\theta(n)^{\dim_{H}(X)}\Big)^2\Big)\Big). \label{General big O}
\end{align}


\end{prop}
\begin{proof}
The quantity $n(I,\theta)$ can be taken to be the smallest $N\in \N$ for which $$Diam(X_{I(z_1,\ldots,z_N)})<Diam(X_{I})\theta(n).$$ It is a consequence of \eqref{Weak Bernoulli diameter} that for each $N\in \mathbb{N}$ we have $$Diam(X_{I(z_1,\ldots,z_N)})= \O( Diam(X_{I})\kappa^{N}).$$  Where $\kappa\in(0,1)$ is some constant independent of $I$ and $z$. Combining these two statements we can prove the following bound
\begin{equation}
\label{n bound}
n(I,\theta)=\O\Big( \log \Big( \frac{1}{\theta(n)}\Big)\Big).
\end{equation}In the derivation of \eqref{n bound} we use the fact that $\theta$ only takes values in the interval $(0,1/2)$. This assumption means we don't need to worry about constants or $\log (\theta(n)^{-1})$ being negative. Applying \eqref{n bound} we have
\begin{equation}
\label{imply}
\sum_{m=n+1}^{n+n(I,\theta)}\mu(X_{I,\theta}\cap E_{m}) \leq n(I,\theta) \mu(X_{I,\theta})=\O\Big(\mu(X_{I,\theta}) \log \Big(\frac{1}{\theta(n)}\Big )\Big).
\end{equation} By similar arguments to those given in Section \ref{Main theorem}, we can combine properties \eqref{Weak Bernoulli measure}--\eqref{cylinder approx} with \eqref{imply} to prove
\begin{equation}
\label{n intersection bounds}
\sum_{m=n+1}^{n+n(I,\theta)}\mu(X_{I,\theta}\cap E_{m}) =\O\Big(\mu(X_{I})\theta(n)^{\dim_{H}(X)} \log \Big(\frac{1}{\theta(n)}\Big )\Big).
\end{equation}
We now obtain an analogue of \eqref{three parts}:
\begin{align}
\sum_{n,m=n(r)}^{Q}\mu(E_{n}\cap E_{m})&=2\sum_{n=n(r)}^{Q}\mu(E_{n})+2\sum_{n=n(r)}^{Q-1}\sum_{m=n+1}^{Q}\mu(E_{n}\cap E_{m})\nonumber\\
&\asymp\sum_{n=n(r)}^{Q}\mu(E_{n})+\sum_{n=n(r)}^{Q-1}\sum_{\substack{I\in\D^{n} \\ (i_{1},\ldots,i_{n(r)})=(y_{1},\ldots,y_{n(r)})}}\sum_{m=n+1}^{Q}\mu(X_{I,\theta}\cap E_{m})\nonumber\\
&\asymp\sum_{n=n(r)}^{Q}\mu(E_{n})+\sum_{n=n(r)}^{Q-1}\sum_{\substack{I\in\D^{n} \\ (i_{1},\ldots,i_{n(r)})=(y_{1},\ldots,y_{n(r)})}}\sum_{m=n+1}^{n+n(I,\theta)}\mu(X_{I,\theta}\cap E_{m})+\nonumber\\
& +\sum_{n=n(r)}^{Q-1}\sum_{\substack{I\in\D^{n} \\ (i_{1},\ldots,i_{n(r)})=(y_{1},\ldots,y_{n(r)})}}\sum_{m=n+n(I,\theta)+1}^{Q}\mu(X_{I,\theta}\cap E_{m})\nonumber\\
&= \sum_{n=n(r)}^{Q}\mu(E_{n})+\O\Big(\sum_{n=n(r)}^{Q-1}\sum_{\substack{I\in\D^{n} \nonumber\\ (i_{1},\ldots,i_{n(r)})=(y_{1},\ldots,y_{n(r)})}}\mu(X_{I}) \theta(n)^{\dim_{H}(X)} \log \Big(\frac{1}{\theta(n)}\Big) \Big)\nonumber\\
&+\O\Big(\sum_{n=n(r)}^{Q-1}\sum_{\substack{I\in\D^{n} \\ (i_{1},\ldots,i_{n(r)})=(y_{1},\ldots,y_{n(r)})}}\sum_{m=n+n(I,\theta)+1}^{Q} \mu(X_{I})\theta(n)^{\dim_{H}(X)}\theta(m)^{\dim_{H}(X)}\Big)\label{Duplication}\\
&\leq  \sum_{n=n(r)}^{Q}\mu(E_{n})+\O\Big(\sum_{n=n(r)}^{Q-1}\sum_{\substack{I\in\D^{n} \\ (i_{1},\ldots,i_{n(r)})=(y_{1},\ldots,y_{n(r)})}}\mu(X_{I}) \theta(n)^{\dim_{H}(X)} \log \Big(\frac{1}{\theta(n)}\Big) \Big)\nonumber\\
&+\O\Big(\sum_{n=n(r)}^{Q}\sum_{\substack{I\in\D^{n} \\ (i_{1},\ldots,i_{n(r)})=(y_{1},\ldots,y_{n(r)})}}\sum_{m=n}^{Q} \mu(X_{I})\theta(n)^{\dim_{H}(X)}\theta(m)^{\dim_{H}(X)}\Big)\label{three split}.
\end{align}In \eqref{Duplication} we used the bounds given by \eqref{n intersection bounds} and \eqref{intersection bound 2}. We now focus on the three terms in \eqref{three split} individually. By identical arguments to those given in Proposition \ref{independence prop}, we have the following bounds for the first and third term:
\begin{equation}
\label{1st term bound}
\sum_{n=n(r)}^{Q}\mu(E_{n})\asymp \mu(X_{(y_{1},\ldots,y_{n(r)})}) \sum_{n=n(r)}^{Q}\theta(n)^{\dim_{H}(X)}
\end{equation}
\begin{equation}
\label{3rd term bound}
\sum_{n=n(r)}^{Q}\sum_{\substack{I\in\D^{n} \\ (i_{1},\ldots,i_{n(r)})=(y_{1},\ldots,y_{n(r)})}}\sum_{m=n}^{Q} \mu(X_{I})\theta(n)^{\dim_{H}(X)}\theta(m)^{\dim_{H}(X)}\asymp \mu(X_{(y_{1},\ldots,y_{n(r)})})\Big(\sum_{n=n(r)}^{Q}\theta(n)^{\dim_{H}(X)}\Big)^2.
\end{equation}
It remains to bound the second term:
\begin{align}
&\sum_{n=n(r)}^{Q-1}\sum_{\substack{I\in\D^{n} \\ (i_{1},\ldots,i_{n(r)})=(y_{1},\ldots,y_{n(r)})}}\mu(X_{I}) \theta(n)^{\dim_{H}(X)} \log \Big(\frac{1}{\theta(n)}\Big)\nonumber\\
&\asymp \mu(X_{(y_{1},\ldots,y_{n(r)})})\sum_{n=n(r)}^{Q-1}\sum_{J\in\D^{n-n(r)} }\mu(X_{J}) \theta(n)^{\dim_{H}(X)} \log \Big(\frac{1}{\theta(n)}\Big) &&(\textrm{By } \eqref{Weak Bernoulli measure})\nonumber\\
&\asymp \mu(X_{(y_{1},\ldots,y_{n(r)})})\sum_{n=n(r)}^{Q-1}\theta(n)^{\dim_{H}(X)} \log \Big(\frac{1}{\theta(n)}\Big).\label{used bound}
\end{align}Substituting \eqref{1st term bound}, \eqref{3rd term bound}, and \eqref{used bound} into \eqref{three split} we obtain \eqref{General big O}.
\end{proof}
Equipped with Proposition \ref{DufSch prop} we are now in a position to prove Theorem \ref{Weak Duffin Schaeffer}.

\begin{proof}[Proof of Theorem \ref{Weak Duffin Schaeffer}]
We assume that $\theta(n)\to 0$ as $n\to \infty$. Consequently, dividing by a positive constant if necessary, we may assume that $\theta$ satisfies the hypothesis of Proposition \ref{DufSch prop}. The case where $\theta$ does not converge to $0$ is easily dealt with. Under this assumption it can be shown that $W(\theta,z)$ is always a set of full $\mu$ measure for any $z\in X$. We omit the details of this fact.

Proposition \ref{DufSch prop} and \eqref{1st term bound} implies that there exists a $C>0$ such that
$$\frac{(\sum_{n=n(r)}^{Q}\mu(E_{n}))^{2}}{\sum_{n,m=n(r)}^{Q}\mu(E_{n}\cap E_m)}$$ can be bounded below by
\begin{equation}
\label{big guy}
\frac{\mu(X_{(y_{1},\ldots,y_{n(r)})})\Big( \sum_{n=n(r)}^{Q}\theta(n)^{\dim_{H}(X)}\Big)^2}{C\Big(\sum_{n=n(r)}^{Q}\theta(n)^{\dim_{H}(X)}+\sum_{n=n(r)}^{Q}\theta(n)^{\dim_{H}(X)} \log \Big(\frac{1}{\theta(n)}\Big)+\Big(\sum_{n=n(r)}^{Q}\theta(n)^{\dim_{H}(X)}\Big)^2 \Big)}.
\end{equation}For $Q$ sufficiently large $$\sum_{n=n(r)}^{Q}\theta(n)^{\dim_{H}(X)}>1$$ and therefore


$$\sum_{n=n(r)}^{Q}\theta(n)^{\dim_{H}(X)}<\Big(\sum_{n=n(r)}^{Q}\theta(n)^{\dim_{H}(X)}\Big)^2.$$ We take the limsup in \eqref{big guy}. It is a consequence of the above remark and our additional assumption \eqref{DufSch Divergence}, that there exists a $C'>0$ such that
\begin{equation*}
\limsup_{Q\to\infty} \frac{(\sum_{n=n(r)}^{Q}\mu(E_{n}))^{2}}{\sum_{n,m=n(r)}^{Q}\mu(E_{n}\cap E_m)}\geq \frac{\mu(X_{(y_{1},\ldots,y_{n(r)})})}{C'}.
\end{equation*}Importantly $C'$ is independent of $y$ and $r$. Proposition \ref{Divergence prop} still holds for this choice of $\theta$, so we may apply Lemma \ref{Erdos lemma}. Applying Lemma \ref{Erdos lemma} we may conclude that $$\mu(\limsup E_n)\geq \frac{\mu(X_{(y_{1},\ldots,y_{n(r)})})}{C'}.$$ Thus \eqref{limsup density3} holds for any $\theta$ satisfying \eqref{DufSch Divergence} and we have completed our proof.
\end{proof}

\subsection{Proof of Theorem \ref{Duffin Schaeffer}}
Assume that $y\in X$ and $r>0$ are all as in the proof of Theorem \ref{conformal theorem}. We will show that for any $z$ satisfying an additional condition, whenever $\theta$ satisfies \eqref{Divergence condition} then \eqref{limsup density3} will hold.

As in the proof of Theorem \ref{Weak Duffin Schaeffer} we need to control $\mu(X_{I,\theta}\cap E_{m})$ for $n<m\leq n+n(I,\theta).$ We start with a simple observation. Let $I\in \D^n$ and $J\in \D^m$ for $n<m\leq n+n(I,\theta).$ Consider the words $I_{\theta}:=(i_1,\ldots,i_n,z_1,\ldots,z_{n(I,\theta)})$ and $J_{\theta}:=(j_1,\ldots,j_m,z_1,\ldots,z_{n(J,\theta)}).$ If $X_{I,\theta}\cap X_{J,\theta}\neq \emptyset$ then either $I_{\theta}$ is a prefix of $J_{\theta},$ or $J_{\theta}$ is a prefix of $I_{\theta},$ i.e.,
\begin{equation}
\label{Case 1}
I_{\theta}=(j_1,\ldots, j_m,z_1,\ldots, z_{n(I,\theta)-(m-n)})
\end{equation} or
\begin{equation}
\label{Case 2}
J_{\theta}=(i_1,\ldots, i_n,z_1,\ldots, z_{n(J,\theta)+(m-n)}).
\end{equation}If \eqref{Case 1} holds then
\begin{equation}
\label{Case 2a}
(z_1,\ldots, z_{n(I,\theta)-(m-n)})=(z_{m-n+1},\ldots, z_{n(I,\theta)}).
\end{equation} Alternatively, if \eqref{Case 2} holds then
\begin{equation}
\label{Case 2b}
(z_{1},\ldots, z_{n(J,\theta)})=(z_{m-n+1},\ldots,z_{n(J,\theta)+(m-n)}).
\end{equation}Consequently, we see that the nonempty intersection of $X_{I,\theta}$ and $X_{J,\theta}$ implies some nontrivial relations for the coding $(z_i).$ This leads to the following definition.

We say that $z$ has a \emph{leading block coding} if there exists a sequence $(z_i)\in \D^{\N},$ such that $(z_i)$ is a coding of $z$ and there exists $l\in \N$ such that the initial word $(z_1,\ldots, z_{l})$ appears only finitely many times in $(z_i).$ Our first step in the proof of Theorem \ref{Duffin Schaeffer} is the following proposition.

\begin{prop}
\label{leading prop}
Suppose $z$ has a leading block coding and $\theta(n)\to 0$ as $n\to\infty$. Then $$\sum_{n,m=n(r)}^{Q}\mu(E_{n}\cap E_{m})=\O\Big(\mu(X_{(y_{1},\ldots,y_{n(r)})})\Big(\sum_{n=n(r)}^{Q}\theta(n)^{\dim_{H}(X)}+\Big(\sum_{n=n(r)}^{Q}\theta(n)^{\dim_{H}(X)}\Big)^2\Big)\Big).$$
\end{prop}
\begin{proof}
Replicating the arguments used in the proof of Proposition \ref{independence prop}, it suffices to show that
\begin{equation}
\label{show equation}
\sum_{n=n(r)}^{Q-1}\sum_{\substack{I\in\D^{n} \\ (i_{1},\ldots,i_{n(r)})=(y_{1},\ldots,y_{n(r)})}}\sum_{m=n+1}^{n+n(I,\theta)}\mu(X_{I,\theta}\cap E_{m})=\O\Big(\mu(X_{(y_{1},\ldots,y_{n(r)})})\sum_{n=n(r)}^{Q}\theta(n)^{\dim_{H}(X)}\Big).
\end{equation}Assume that $(z_1,\ldots,z_l)$ appears only finitely many times in the coding $(z_i)$. Since $\theta(n)\to 0$ we must have $n(I,\theta)\to \infty$ as $|I|\to\infty.$ It follows that for $n$ sufficiently large, if $X_{I,\theta}\cap X_{J,\theta}\neq \emptyset$ and \eqref{Case 2b} occurs, then $$(z_1,\ldots,z_l)=(z_{m-n+1},\ldots,z_{m-n+l}).$$ This can only occur finitely many times by definition. Suppose $X_{I,\theta}\cap X_{J,\theta}\neq \emptyset$ and \eqref{Case 2a} occurs, then if $n<m\leq n+n(I,\theta)-l$ we have $$(z_1,\ldots, z_{l})=(z_{m-n+1},\ldots, z_{m-n+l}).$$ Which can only occur finitely many times by definition. As there at most $l-1$ values of $m$ satisfying $n+n(I,\theta)-l<m\leq n+n(I,\theta),$ it follows that $$\sup_{n\in\mathbb{N}}\sup_{I\in \D^n}\#\Big\{m\in\mathbb{N}:n<m\leq n+n(I,\phi)\textrm{ and }X_{I,\theta}\cap E_m\neq \emptyset\Big\}<\infty.$$Therefore,
\begin{align*}
\sum_{n=n(r)}^{Q-1}\sum_{\substack{I\in\D^{n} \\ (i_{1},\ldots,i_{n(r)})=(y_{1},\ldots,y_{n(r)})}}\sum_{m=n+1}^{n+n(I,\theta)}\mu(X_{I,\theta}\cap E_{m})&=\O\Big(\sum_{n=n(r)}^{Q-1}\sum_{\substack{I\in\D^{n} \\ (i_{1},\ldots,i_{n(r)})=(y_{1},\ldots,y_{n(r)})}}\mu(X_{I,\theta})\Big)\\
&=\O \Big(\mu(X_{(y_{1},\ldots,y_{n(r)})})\sum_{n=n(r)}^{Q-1}\sum_{J\in\D^{n-n(r)}}\mu(X_{J})\theta(n)^{\dim_{H}(X)}\Big)\\
&=\O \Big(\mu(X_{(y_{1},\ldots,y_{n(r)})})\sum_{n=n(r)}^{Q}\theta(n)^{\dim_{H}(X)}\Big).
\end{align*}The proof of the second equality above follows from the same arguments used in the proof of Theorem \ref{Weak Duffin Schaeffer}. Thus we have proved \eqref{show equation} and our proof is complete.

\end{proof} By an analogous argument to that given in the proof of Theorem \ref{conformal theorem}, we may combine Proposition \ref{leading prop} with Lemma \ref{Erdos lemma} and Proposition \ref{Divergence prop}, to conclude that if $z$ has a leading block coding and $\theta$ is any function satisfying $\theta(n)\to 0$ as $n\to\infty$ and \eqref{Divergence condition}, then \eqref{limsup density3} holds and consequently $\mu$ almost every $x$ is an element of $W(\theta,z)$. As we previously remarked upon, the case where $\theta$ does not converge to $0$ is easily dealt with. Under this assumption we always have that $W(\theta,z)$ is a set of full $\mu$ measure for any $z\in X$. Summarising the above, we have the following result.
\begin{prop}
\label{leading block 1}
Suppose $z$ has a leading block coding and $\theta:\mathbb{N}\to\mathbb{R}_{\geq 0}$ satisfies \eqref{Divergence condition}, then $\mu$ almost every $x$ is an element of $W(\theta,z)$.
\end{prop}
To prove Theorem \ref{Duffin Schaeffer} we need to prove that the set of $z$ with a leading block coding has full Hausdorff dimension.

\begin{prop}
\label{leading block 2}
$\dim_{H}(\{z\in X:z \textrm{ has a leading block coding}\})=\dim_{H}(X)$.
\end{prop}
\begin{proof}
Let $Y:=\{z\in X:z \textrm{ has a leading block coding}\}$. Let $N\in\N$ and let us fix a digit $i\in \D$. Consider the word $i^{N}=(i,\ldots,i)\in \D^{N}$ and the IFS $\{\phi_{I}\}_{\substack{I\in \D^{N}\setminus \{i^{N}}\}}$ with corresponding attractor $X^{i,N}$. For any $\epsilon >0$ we can pick $N$ sufficiently large that
\begin{equation}
\label{BiLipschitz}
\dim_{H}(X)-\epsilon< \dim_{H}(X^{i,N})\leq \dim_{H}(X).
\end{equation}Now consider the set $\phi_{i^{2N}}(X^{i,N}),$ where $i^{2N}\in \D^{2N}$ is the word consisting of $2N$ consecutive $i$'s. Every element of $\phi_{i^{2N}}(X^{i,N})$ has a coding that begins with $i^{2N}$ and for which $i^{2N}$ occurs only finitely many times. This is because every sequence in $(\D^{N}\setminus \{i^{N}\})^{\mathbb{N}}$ cannot contain the word $i^{2N}$.  Therefore $\phi_{i^{2N}}(X^{i,N})\subseteq Y.$ Since $\phi_{i^{2N}}$ is a bi-Lipschitz map it follows from \eqref{BiLipschitz} that $$\dim_{H}(X)-\epsilon<\dim_{H}(\phi_{i^{2N}}(X^{i,N}))\leq \dim_{H}(X).$$ Consequently, $$\dim_{H}(X)-\epsilon<\dim_{H}(Y)\leq \dim_{H}(X).$$ Since $\epsilon$ is arbitrary we must have $\dim_{H}(Y)=\dim_{H}(X)$ as required.
\end{proof}
Combining Proposition \ref{leading block 1} and Proposition \ref{leading block 2} we may conclude Theorem \ref{Duffin Schaeffer}.

\section{The mass transference principle and some applications}
\label{mass transference}
Suppose $\Psi:\D^{*}\to\mathbb{R}$ is such that $\sum_{n=1}^{\infty}\sum_{I\in \D^{n}}\Psi(I)^{\dim_{H}(X)}<\infty$, then by Theorem \ref{conformal theorem} we know that $\mathcal{H}^{\dim_{H}(X)}(W(\Psi,z))=0$ for any $z\in X$. Under this assumption it is natural to ask what is the Hausdorff dimension of $W(\Psi,z)$. In this section we obtain a partial solution to this question by employing the mass transference principle introduced by Beresnevich and Velani \cite{BerVel}. We end this section with some applications of this result.

\subsection{The mass transference principle}
The first result of this section is the following theorem.
\begin{thm}
\label{Hausdorff theorem}Let $\Phi$ be an iterated function system with attractor $X$.
\begin{enumerate}
  \item Let $\Psi:\D^{*}\to\mathbb{R}_{\geq 0}$ and $s> 0$. Suppose $$\sum_{n=1}^{\infty} \sum_{I\in\D^{n}} \Psi(I)^{s}<\infty,$$ then $\mathcal{H}^{s}(W(\Psi,z))=0$ for all $z\in X$.
  \item If $\Phi$ is a conformal iterated function system satisfying the open set condition and $\theta:\mathbb{N}\to\mathbb{R}_{\geq 0}$ is a decreasing function that satisfies \eqref{Divergence condition},
  then $$\mathcal{H}^{\dim_{H}(X)/t}(W((Diam(X_{I})\theta(|I|))^{t},z))=\mathcal{H}^{\dim_{H}(X)/t}(X)$$ for all $z\in X$ and $t\geq 1$.
\end{enumerate}
\end{thm}


Taking $t=1$ in the above we observe that statement $2$ from Theorem \ref{conformal theorem} is a consequence of Theorem \ref{Hausdorff theorem}. It is in fact the case that the opposite implication is true. Theorem \ref{conformal theorem} implies Theorem \ref{Hausdorff theorem}. This is because of the mass transference principle. We now briefly detail this technique.

Let $X\subset \R^{d}$ and assume that $X$ is Ahlfors regular. Given $s>0$ and a ball $B=B(x,r),$ we let $B^s=B(x,r^{s/\dim_{H}(X)}).$ The following theorem is a simplified version of Theorem $3$ proved in \cite{BerVel}. It will allow us to prove that Theorem \ref{conformal theorem} implies Theorem \ref{Hausdorff theorem}.

\begin{thm}
\label{Mass Transference theorem}
Let $X$ be as above and $(B_{l})_{l=1}^{\infty}$ be a sequence of balls in $X$ with radii tending to zero. Let $s>0$ and suppose that for any ball $B$ in $X$ we have $$\mathcal{H}^{\dim_{H}(X)}(B\cap \limsup_{l \to \infty}B_l^s)=\mathcal{H}^{\dim_{H}(X)}(B).$$ Then, for any ball $B$ in $X$ $$\mathcal{H}^{s}(B\cap \limsup_{l \to \infty}B_l)=\mathcal{H}^{s}(B).$$
\end{thm}

We now prove Theorem \ref{Hausdorff theorem} using Theorem \ref{Mass Transference theorem}.

\begin{proof}[Proof of Theorem \ref{Hausdorff theorem}]

The proof of statement $1$ is analogous to the proof of statement $1$ from Theorem \ref{conformal theorem} so we omit it.

Suppose $\Phi$ is a conformal iterated function system satisfying the open set condition. Assume $\theta:\mathbb{N}\to\R_{\geq 0}$ is a decreasing function satisfying \eqref{Divergence condition}. Let $z\in X$, $t\geq 1,$ and $(B_{l})_{l=1}^{\infty}$ be an enumeration of the set of balls $\{B(\phi_{I}(z),(Diam(X_I)\theta(|I|))^t)\}_{I\in D^{*}}.$ It follows from the definition that $(B^{\dim_{H}(X)/t}_{l})_{l=1}^{\infty}$ is an enumeration of the set of balls $\{B(\phi_{I}(z),Diam(X_I)\theta(|I|))\}_{I\in \D^{*}}.$ In which case, by Theorem \ref{conformal theorem} we know that $$\mathcal{H}^{\dim_{H}(X)}(B\cap \limsup_{l \to \infty}B_l^{\dim_{H}(X)/t})=\mathcal{H}^{\dim_{H}(X)}(B)$$ for any ball $B$. Applying Theorem \ref{Mass Transference theorem} we may conclude that $$\mathcal{H}^{\dim_{H}/t}(B\cap \limsup_{l \to \infty}B_l)=\mathcal{H}^{\dim_{H}(X)/t}(B),$$ for any ball $B$, and $$\mathcal{H}^{\dim_{H}(X)/t}(W((Diam(X_I)\theta(|I|))^t,z))=\mathcal{H}^{\dim_{H}(X)/t}(X).$$
\end{proof}
Statement $2$ from Theorem \ref{Hausdorff theorem} implies a lower bound for the Hausdorff dimension of certain limsup sets. It is a consequence of statement $1$ that this bound is in fact optimal and implies the following result.
\begin{cor}
\label{dimension equality}

If $\Phi$ is a conformal iterated function system satisfying the open set condition and $\theta:\mathbb{N}\to\mathbb{R}_{\geq 0}$ is a decreasing function that satisfies \eqref{Divergence condition},
  then $$\dim_{H}(W((Diam(X_{I})\theta(|I|))^{t},z))=\frac{\dim_{H}(X)}{t}$$ for all $z\in X$ and $t\geq 1$
\end{cor}We leave the proof of Corollary \ref{dimension equality} to the interested reader.
\begin{remark}
One would like to improve Theorem \ref{Hausdorff theorem} to a more general statement that covered arbitrary dimension functions. Such a result appeared in \cite{LSV}. However, since the radii of our balls take the restrictive form $Diam(X_{I})\theta(|I|)$, we cannot prove such a general statement. In the more general case, the analogue of the normalised ball $B^s$ does not take a form we can work with.

\end{remark}

\subsection{Applications}
We now include some applications of Theorem \ref{Hausdorff theorem}. Before giving these applications it is useful to build some general theory.

Let $E$ be some subset of $\mathbb{R}^{d}$. We call a function $H:E\to\mathbb{R}_{>0}$ a \textit{height function} on $E$. We say that an IFS $\Phi$ \emph{respects $E$ and $H$ }if the following two conditions hold:
\begin{itemize}
\item $\phi(x)\in E$ for all $x\in E$ and $\phi\in \Phi$.
\item There exists $C>1$ such that $H(\phi(x))<CH(x)$ for all $x\in E$ and $\phi\in \Phi$.
\end{itemize}

\begin{prop}
\label{height prop}
Let $\Phi$ be a conformal iterated function system satisfying the open set condition and $H$ be a height function on a set $E$. Suppose $E\cap X\neq \emptyset$ and $\Phi$ respects $E$ and $H,$ then for any $l>0$ we have
    $$\dim_{H}(X)\Big(\Big\{x\in X: |x-e|<H(e)^{-l} \textrm{ for i.m. } e\in E\Big\}\Big)>0.$$
\end{prop}
\begin{proof}
Fix $\Phi$, $H,$ $E$ and $l$. Let $a\in E\cap X.$ We will prove that
\begin{equation}
\label{t inclusion}
W(Diam(X_I)^{t},a)\subseteq \Big\{x\in X: |x-e|<H(e)^{-l} \textrm{ for i.m. } e\in E\Big\}
\end{equation}for some suitable choice of $t\geq 1$. Our result will then follow from Corollary \ref{dimension equality}. Since each element of $\Phi$ is a contraction there exists $\gamma\in(0,1)$ such that
\begin{equation}
\label{Diameter decaying}
Diam(X_I)<\gamma^{|I|}Diam(X)
 \end{equation}for all $I\in\D^{*}$. Since $\Phi$ respects $E$ and $H$ we have
\begin{equation}
\label{height growing}
 H(\phi_{I}(a))\leq C^{|I|}H(a)
 \end{equation} for all $I\in \D^{*}$. Let us now choose $t$ sufficiently large that
\begin{equation}
\label{n large}
(\gamma^{n}Diam(X))^t<\Big(\frac{1}{C^{n}H(a)}\Big)^l,
\end{equation}for all $n$ sufficiently large. As a consequence of \eqref{Diameter decaying}, \eqref{height growing} and \eqref{n large} the following inclusions hold for $|I|$ sufficiently large
\begin{equation*}
B(\phi_{I}(a),Diam(X_I)^{t})\subseteq B(\phi_{I}(a),(\gamma^{|I|}Diam(X))^t) \subseteq B(\phi_{I}(a),(C^{|I|}H(a))^{-l})\subseteq B(\phi_{I}(a),H(\phi_{I}(a))^{-l}).
\end{equation*} These inclusions along with the fact $\phi_{I}(a)\in E$ for all $I\in \D^{*}$ imply \eqref{t inclusion}.
\end{proof}

Proposition \ref{height prop} allows us to deduce the existence of elements within our attractor $X$ that are ``very well approximable," where the definition of ``very well approximable" may be defined in a way that is independent from our attractor $X$. In the next section we give two applications which exhibit this phenomenon.

Note that in the proof of Proposition \ref{height prop} we didn't require the full strength of Corollary \ref{dimension equality}. One can prove this result by combining elementary arguments with the mass transference principle. Therefore, Proposition \ref{height prop} and the applications below should be interpreted as consequence of the mass transference principle within the setting of IFS's, rather than as a consequence of approximation regularity.

\subsubsection{A problem of Mahler}
A classical theorem due to Dirichlet states that for any $\alpha=(\alpha_1,\ldots,\alpha_d)\in \R^{d}$ and $Q\in \N,$ there exists $(p_1,\ldots,p_d)\in \mathbb{Z}^d$ and $1\leq q<Q^d$ such that $$\max_{1\leq i \leq d}|q\alpha_i-p_i|\leq \frac{1}{Q}.$$ This theorem implies that for any $(\alpha_1,\ldots,\alpha_d)\in \R^{d}$ with at least one $\alpha_i$ irrational, there exists infinitely many $(p_{1}/q,\ldots,p_{d}/q)\in \mathbb{Q}^d$ such that $$\max_{1\leq i \leq d}\Big|\alpha_i-\frac{p_i}{q}\Big|\leq \frac{1}{q^{1+1/d}}.$$ We call $\alpha\in \R^{d}$ \emph{very well approximable} if there exists $\tau>1+1/d$ for which there exists infinitely many $(p_{1}/q,\ldots,p_{d}/q)\in \mathbb{Q}^d$ satisfying
\begin{equation}
\label{well approximable}
\max_{1\leq i \leq d}\Big|\alpha_i-\frac{p_i}{q}\Big|\leq \frac{1}{q^{\tau}}.
\end{equation} We call $\alpha$ \emph{Liouville} if it satisfies \eqref{well approximable} for all $\tau>1+1/d$. It is well known that the set of very well approximable numbers has Lebesgue measure zero, and the set of Liouville numbers has Hausdorff dimension zero. In \cite{Bug} the following assertion was made and attributed to Mahler.
\\

\noindent \emph{There exists very well approximable numbers, other than Liouville numbers, in the middle third Cantor set.}
\\

This assertion was proved to be correct in \cite{LSV}. Related work appears in \cite{FishSim}. Applying Proposition \ref{height prop} we now obtain a more general version of this result. We call $\Phi$ a \emph{rational iteration function system} if each element of $\Phi$ is a similarity of the form $\phi_{i}(x)=\frac{p_i x}{q_i}+(\frac{a_{i,1}}{b_{i,1}},\ldots,\frac{a_{i,d}}{b_{i,d}})$ where $p_i,q_i \in\mathbb{Z}$ and $a_{i,1},b_{i,1},\ldots,a_{i,d},b_{i,d}\in \mathbb{Z}$.
\begin{thm}
\label{Mahler}
Let $\Phi$ be a rational iterated function system satisfying the open set condition, then $X$ contains very well approximable numbers that are not Liouville.
\end{thm}
\begin{proof}
By Proposition \ref{height prop} and the fact that the Liouville numbers have Hausdorff dimension zero, it suffices to show that $X\cap\mathbb{Q}^d\neq\emptyset$ and $\Phi$ respects $\mathbb{Q}^d$ and $H$. Where $H:\mathbb{Q}^d\to\mathbb{N}$ is the height function defined to be $H((\frac{p_{1}}{q_{1}},\ldots, \frac{p_d}{q_d}))=\textrm{lcm}(q_{1},\ldots,q_{d}),$ where in this expression each $p_i/q_i$ is assumed to be in its reduced form and $\textrm{lcm}$ denotes the lowest common multiple.

To see that $X\cap \mathbb{Q}^d\neq\emptyset,$ we remark that for any $\phi_{i}\in \Phi$ the unique fixed point satisfying $\phi_{i}(x)=x$ is contained in $X$. Proving that this fixed point is contained in $\mathbb{Q}^{d}$ follows immediately from the definition of a rational iterated function system.

We now show that $\Phi$ respects $\mathbb{Q}^d$ and $H.$ Clearly any element of $\Phi$ maps $\mathbb{Q}^d$ to $\mathbb{Q}^d$. It remains to show that we satisfy the required growth condition. Fix $\overline{x}=(\frac{x_{1}}{y_{1}},\ldots, \frac{x_d}{y_{d}})\in\mathbb{Q}^{d}.$ Let $y=\textrm{lcm}(y_1,\ldots,y_m)$, by an abuse of notation we write $\overline{x}=(\frac{x_{1}}{y},\ldots, \frac{x_d}{y}).$ We observe the following
\begin{align}
\phi_i(\overline{x})&=\Big(\frac{p_{i}x_{1}}{q_{i}y},\ldots, \frac{p_{i}x_d}{q_{i}y}\Big)+\Big(\frac{a_{i,1}}{b_{i,1}},\ldots,\frac{a_{i,d}}{b_{i,d}}\Big)\nonumber\\
&= \Big(\frac{b_{i,1}p_{i}x_{1}+a_{i,1}q_{i}y}{b_{i,1}q_{i}y},\ldots, \frac{b_{i,d}p_{i}x_{d}+a_{i,d}q_{i}y}{b_{i,d}q_{i}y}\Big)\label{lcm bound}.
\end{align}Each term in \eqref{lcm bound} can be rewritten to have denominator $q_{i}y\cdot \prod_{j=1}^{d} b_{i,j}.$ Therefore $$H(\phi_i(\overline{x}))\leq q_{i}y\cdot \prod_{j=1}^{d} b_{i,j}.$$ Taking $C=\sup_{i\in \D}q_{i}\cdot \prod_{j=1}^{d} b_{i,j}$ we have $$H(\phi_{i}(\overline{x}))\leq C H(\overline{x}),$$ for all $\overline{x}\in\mathbb{Q}^d$ and $\phi_i\in \Phi$. Thus $\Phi$ respects $\mathbb{Q}^d$ and $H$.
\end{proof}

\subsubsection{Approximating badly approximable numbers by quadratic irrationals}
For our next application we prove that there exist badly approximable numbers that are ``very well approximated" by quadratic irrationals. We start by recalling what it means to be badly approximable and give an overview of approximation by algebraic numbers.

We call $x\in (0,1)$ \emph{badly approximable} if there exists $\kappa(x)>0$ such that
$$\Big|x-\frac{p}{q}\Big|>\frac{\kappa(x)}{q^2} \textrm{ for all }(p,q)\in \mathbb{Z}\times \mathbb{N}.$$ Every $x\in(0,1)\setminus \mathbb{Q}$ has a unique continued fraction expansion, that is a unique sequence $(a_n(x))\in \mathbb{N}^{\mathbb{N}}$ such that
$$x=\cfrac{1}{a_1+\cfrac{1}{a_2+\cfrac{1}{\cdots+\cdots}}}:=\lim_{n\to\infty}\cfrac{1}{a_1+\cfrac{1}{\cdots+\cfrac{1}{a_n}}}.$$ We refer to the sequence $(a_n)$ as the partial quotients of $x$. Badly approximable numbers are characterised by their continued fraction expansion. It is well known that $x$ is badly approximable if and only if its sequence of partial quotients is bounded \cite{Bug}.

We now detail some of the background behind approximation by algebraic numbers. What follows is taken from \cite{Bug}. The height of an algebraic number $\alpha,$ denoted by $H(\alpha),$ is the maximum of the moduli of the coefficients of its minimal polynomial. For example, we have $H(\sqrt{2})=2$ since the minimal polynomial of $\sqrt{2}$ is $x^2-2$. The degree of an algebraic number $\alpha$, denoted by $\textrm{deg}(\alpha),$ is the degree of the minimal polynomial.

Given $x\in \mathbb{R}$ and $n\in\mathbb{N},$ we define $w_{n}(x)$ to be the supremum of the real numbers $\omega$ for which there exists i.m. real algebraic numbers $\alpha$ with $\textrm{deg}(\alpha)\leq n$ satisfying $$0<|x-\alpha|<H(\alpha)^{-\omega-1}.$$ We then let $$w(x):=\limsup_{n\to\infty}\frac{w_{n}(x)}{n}.$$ We will not use the quantity $w(x).$ We merely remark that it can be used to give a classification of the real numbers in terms of how transcendental they are. This classification is known as Koksma's classification \cite{Koksma}. The following result describes the generic behaviour of $w_{n}(x).$
\begin{thm}
\label{generic algebraic}
Lebesgue almost every $x\in \R$ satisfies $w_{n}(x)=n$ for each $n\in\mathbb{N}$.
\end{thm}For a proof of Theorem \ref{generic algebraic} see \cite{Bug}. The proof of this theorem is originally due to Sprind\v{z}uk \cite{Spr1,Spr2}. For our applications we will only require Theorem \ref{generic algebraic} when $n=2.$

We call $x\in \mathbb{R}$ \emph{quadratically very well approximable} if $w_{2}(x)>2.$ We call $x\in \mathbb{R}$ \emph{quadratically Liouville} if $w_{2}(x)=\infty$. It is a consequence of a result due to Kasch and Volkmann \cite{KaVo} that the set of quadratically Liouville numbers has Hausdorff dimension zero. The following theorem demonstrates that there are badly approximable numbers that are quadratically very well approximable but not quadratically Liouville.

\begin{thm}
Let $D\subset\mathbb{N}$ be a finite set that contains at least two elements and $X_{D}:=\{x\in(0,1)\setminus \mathbb{Q}: (a_n(x))\in D^{\mathbb{N}}\}.$ Then $X_{D}$ contains numbers that are quadratically very well approximable but not quadratically Liouville.
\end{thm}
\begin{proof}
We start our proof by remarking that $X_{D}$ can be identified with the unique attractor for the IFS $\Phi=\{\phi_{i}\}_{i\in D}$ where $\phi_{i}(x)=\frac{1}{x+i}.$ Moreover, this $\Phi$ is a conformal iterated function system satisfying the open set condition. Therefore we can use Proposition \ref{height prop}. Let $E$ be the set of quadratic irrationals. We know by the aforementioned result of Kasch and Volkmann that the set of quadratically Liouville numbers has Hausdorff dimension zero. Therefore to complete our proof it suffices to show that we satisfy the remaining hypothesis of Proposition \ref{height prop}. Namely we need to show that $E\cap X_{D}\neq \emptyset$ and $\Phi$ respects $E$ and $H$. We have $E\cap X_{D}\neq \emptyset$ since every eventually periodic  $(a_{n})\in D^{\mathbb{N}}$ is the continued fraction expansion of a quadratic irrational. It remains to show that $\Phi$ respects $E$ and $H.$ Fix $\alpha\in E$ with minimal polynomial $ax^2+bx+c.$ Then $H(\alpha)=\max\{|a|,|b|,|c|\}.$ A simple calculation shows that $\phi_{i}(\alpha)$ is a root of the polynomial \begin{equation}
\label{new equation}
(ai^2-bi+c)x^2+(b-2ia)x+a.
\end{equation} Therefore $\phi_{i}(\alpha)$ is either a quadratic irrational or a rational number. However, if $\phi_{i}(\alpha)$ is rational then it can be shown that $\alpha$ is also rational, a contradiction. Therefore each element of $\Phi$ maps $E$ to $E$ and it remains to show we satisfy the growth condition. Since $\phi_{i}(\alpha)$ is a quadratic irrational, \eqref{new equation} cannot be factorised into two linear factors. Consequently, \eqref{new equation} can be written in the form $nM(x),$ where $n\in\mathbb{Z}\setminus\{0\}$ and $M(x)\in \mathbb{Z}[x]$ is the minimal polynomial of $\phi_{i}(\alpha)$. It follows that
\begin{align*}
H(\phi_{i}(\alpha))&\leq \max\{|ai^2-bi+c|, |b-2ia|, |a|\}\\
&\leq \max\{(|a|+|b|+|c|)i^2, (|b|+2|a|)i, |a|\}\\
&\leq 3 i^2 H(\alpha).
\end{align*}Taking $C=\max_{i\in D}\{3i^2\}$ we have $$H(\phi_{i}(\alpha))\leq C H(\alpha)$$ for all $i\in D$ and $\alpha\in E.$ Therefore $\Phi$ respects $E$ and $H$ and our proof is complete.
\end{proof}

\section{The overlapping case and further directions}
\label{final section}
We conclude this paper by proving that whenever our IFS contains an exact overlap there are no approximation regular pairs. We also include some discussion of the overlapping case and suggest some future directions.

Understanding the structure of overlapping attractors is a classical problem. For $\Phi$ a conformal iterated function system we always have $\dim_{H}(X)\leq \dim_{S}(X).$ We also trivially have the upper bound $\dim_{H}(X)\leq d$ where $d$ is the dimension of the ambient Euclidean space. These two bounds imply
\begin{equation}
\label{dimensionbounds}
\dim_{H}(X)\leq \min\{\dim_{S}(X),d\}.
\end{equation}Determining conditions under which we have equality in \eqref{dimensionbounds} is an active area of research, see \cite{Hoc,PerSol} and the references therein. There is a standard way of constructing examples for which we have strict inequality in \eqref{dimensionbounds}. We construct an IFS in such a way that there is an exact overlap, i.e., there exists $I,J\in\mathcal{D}^{n}$ such that $\phi_{I}=\phi_{J}$. This means we can remove one of these maps from our IFS and still be left with the same attractor. This new IFS has a strictly smaller similarity dimension which can lead to a strict inequality in \eqref{dimensionbounds}. It is conjectured that exact overlaps are the only mechanism by which we can have strict inequality in \eqref{dimensionbounds}. Interestingly, as far as the author knows, the only known condition which can result in a pair $(\Phi,z)$ failing to be approximation regular is when there is an exact overlap. This result is proved in the following theorem.

\begin{thm}
\label{overlap}
Let $\Phi$ be a conformal iterated function system containing an exact overlap. Then there are no approximation regular pairs.
\end{thm}
\begin{proof}
Suppose that $I,J\in \D^{k}$ are such that $\phi_{I}=\phi_{J}$. Let us fix $z\in X.$ It suffices to construct a function $\theta:\mathbb{N}\to\mathbb{R}_{\geq 0}$ such that we have divergence in \eqref{Divergence condition general}, yet
\begin{equation}
\label{dimension drop}
\dim_{H}(W(Diam(X_I)^{\dim_{S,\Phi}(X)/\dim_{H}(X)}\theta(|I|),z))<\dim_{H}(X).
\end{equation} Note that in \eqref{dimension drop} we have included the subscript $\Phi$ in the similarity dimension. This is to emphasise the dependence on $\Phi$. This dependence will be important in what follows.

Let $\theta(|I|)=1$ for all $I\in \D^{*},$ so our approximating function is simply $\Psi(I)=Diam(X_I)^{\dim_{S,\Phi}(X)/\dim_{H}(X)}$. The summation in \eqref{Divergence condition general} reduces to $$\sum_{n=1}^{\infty}\sum_{I\in \D^{n}}Diam(X_I)^{\dim_{S,\Phi}(X)}.$$ It is a consequence of Bowen's equation and \eqref{Pressure dimension} that this series diverges. It remains to show that \eqref{dimension drop} holds.

The attractor $X$ is also the attractor for the iterated function system $\Phi^{k}:=\{\phi_I\}_{I\in\D^{k}\setminus\{J\}}.$ This follows by iterating \eqref{Hutchinson's equation} and using the fact that we have an exact overlap. $\Phi^{k}$ is a conformal iterated function system, so we can consider it's similarity dimension $\dim_{S,\Phi^{k}}(X)$. As a consequence of the exact overlap we have $$\dim_{S,\Phi^{k}}(X)<\dim_{S,\Phi}(X).$$

Now let us choose $\epsilon>0$ such that
\begin{equation}
\label{extra decay}
\dim_{S,\Phi^{k}}(X)<\frac{\dim_{S,\Phi}(X)}{\dim_{H}(X)}\cdot (\dim_{H}(X)-\epsilon)
\end{equation}
Replicating arguments from the proof of statement $1$ from Theorem \ref{conformal theorem}, we can show that for any $\rho>0$ there exists $M\in \mathbb{N}$ sufficiently large such that we have the following bound
\begin{equation}
\label{measure bound}
\mathcal{H}^{\dim_{H}(X)-\epsilon}_{\rho}(W(Diam(X_I)^{\frac{\dim_{S,\Phi}(X)}{\dim_{H}(X)}},z))
\leq \sum_{n=M}^{\infty}\sum_{\substack{I\in \D^{n}\\ J \textrm{ is not a subword of } I}}Diam(X_I)^{\frac{\dim_{S,\Phi}(X)}{\dim_{H}(X)}\cdot (\dim_{H}(X)-\epsilon)}.
\end{equation}It is a consequence of \eqref{extra decay} and \eqref{Pressure dimension} that the second summation on the right hand side of \eqref{measure bound} tends to zero exponentially fast. Therefore the right hand side converges and $M$ can be chosen to make this summation arbitrarily small. It follows that $\mathcal{H}^{\dim_{H}(X)-\epsilon}_{\rho}(W(Diam(X_I)^{\frac{\dim_{S,\Phi}(X)}{\dim_{H}(X)}},z))=0,$ and since $\rho$ is arbitrary we must have $\mathcal{H}^{\dim_{H}(X)-\epsilon}(W(Diam(X_I)^{\frac{\dim_{S,\Phi}(X)}{\dim_{H}(X)}},z))=0.$ Thus \eqref{dimension drop} holds and our proof is complete.
\end{proof}
Theorem \ref{overlap} and the discussion at the start of this section give rise to several natural questions.

\begin{enumerate}
\item Is the only condition under which an IFS fails to be approximation regular when there is an exact overlap?
\item Can one relate the approximation regularity properties of a conformal IFS to other nice properties of the attractor? For example, can it be related to equality in \eqref{dimensionbounds}? Can it be related to the absolute continuity of certain natural measures supported on $X$?
\item Another natural direction to pursue is to consider more general attractors. In particular, one can ask what is the analogue of the above theory in the setting of self-affine sets, IFS's consisting of infinitely many contracting maps, and for randomly defined attractors. We currently have no results in this direction. We expect that for a self-affine set the analogue of the divergence condition \eqref{Divergence condition general} would have to take into account the rotations that might be present with the IFS. This was something we didn't have to consider for conformal iterated function systems.
\end{enumerate}

We believe that approximation regularity can be used as an effective tool to measure how much an attractor overlaps. It is possible however that such an approach is to blunt and a more subtle approach is required. Expecting the divergence of certain sums to be the deciding criteria in determining whether a limsup set has full measure/Hausdorff dimension could be wishful thinking. Instead of looking solely at the divergence of certain sums, one should perhaps put a greater emphasis on determining those approximating functions $\Psi$ for which $W(\Psi,z)$ is of full measure/Hausdorff dimension.
\\

\noindent \textbf{Acknowledgements.} The author would like to thank the anonymous referee whose comments improved the exposition of this article.

\end{document}